\documentclass[reqno,11pt, twoside]{amsart}
\usepackage{latexsym}
\usepackage{amsfonts}
\usepackage{amssymb}
\usepackage{mathrsfs}
\usepackage[all]{xy}
\usepackage{bbm}
\usepackage{xcolor}
\usepackage{ulem}
\usepackage[mathscr]{euscript}
\usepackage{dsfont}
\usepackage{float}

\usepackage[english]{babel}

\usepackage{amsmath}
\usepackage{graphicx}
\usepackage[colorinlistoftodos]{todonotes}
\usepackage[colorlinks=true, allcolors=blue]{hyperref}
\usepackage{tikz-cd}
\usepackage{tikz}
\usepackage{pst-node}
\usepackage{mathtools}
\usetikzlibrary{graphs}

\usepackage[T2A]{fontenc}

\theoremstyle{plain}
\newtheorem{thm}{Theorem}[section]
\newtheorem{claim}{Claim}[thm]
\newtheorem{lem}[thm]{Lemma}
\newtheorem{rem}[thm]{Remark}
\newtheorem{prop}[thm]{Proposition}
\newtheorem{cor}[thm]{Corollary}

\theoremstyle{definition}
\newtheorem{defn}[thm]{Definition}
\newtheorem{exmp}[thm]{Example}
\newtheorem{ques}[thm]{Question}

\theoremstyle{definition}

\theoremstyle{remark}

\numberwithin{equation}{section}

\newcommand{\SL}{\operatorname{SL}}

\newcommand{\dist}{\operatorname{dist}}

\newcommand\smvee{\raise0.9ex\hbox{$\scriptscriptstyle\vee$}}

\newcommand\ca[1]{\mathcal{#1}}

\newcommand{\sm}{\smallsetminus}

\newcommand{\R}{{\mathbb{R}}}
\newcommand{\T}{{\mathbb{T}}}
\newcommand{\Co}{{\mathbb{C}}}
\newcommand{\Z}{{\mathbb{Z}}}
\newcommand{\N}{{\mathbb{N}}}

\newcommand{\ve}{{\bf e}}

\newcommand{\ignore}[1]{{}}

\newcommand {\new}[1]   {\textcolor{orange}{#1}}

\newcommand{\al}{\alpha}
\newcommand{\ga}{\gamma}
\newcommand{\Ga}{\Gamma}
\newcommand{\del}{\delta}
\newcommand{\Del}{\Delta}
\newcommand{\lam}{\lambda}

\newcommand{\eps}{\epsilon}

\newcommand{\cA}{\mathcal{A}}

\newcommand{\cC}{\mathcal{C}}

\newcommand{\cI}{\mathcal{I}}

\newcommand{\cM}{\mathcal{M}}

\newcommand{\cP}{\mathcal{P}}

\newcommand{\bC}{\mathbb{C}}

\newcommand{\bR}{\mathbb{R}}
\newcommand{\bZ}{\mathbb{Z}}
\newcommand{\bQ}{\mathbb{Q}}

\newcommand{\bN}{\mathbb{N}}

\newcommand{\bT}{\mathbb{T}}

\newcommand\norm[1]{\|#1\|}

\newcommand\wh[1]{\widehat{#1}}
\newcommand\seta[1]{\left\{#1\right\}}
\newcommand\pa[1]{\left(#1\right)}
\newcommand\idist[1]{\langle#1\rangle}

\newcommand\av[1]{|#1|}
\newcommand\on[1]{\operatorname{#1}}

\newcommand\mb[1]{\mathbf{#1}}

\newcommand\mat[1]{\pa{\begin{matrix}#1\end{matrix}}}
\newcommand\br[1]{\left[#1\right]}
\newcommand\smallmat[1]{\pa{\begin{smallmatrix}#1\end{smallmatrix}}}

\usepackage{todonotes}
\usepackage{setspace}
\newcounter{mycomment}
\newcommand{\mycomment}[2][]{
\refstepcounter{mycomment}
{
\setstretch{0.7}
\todo[color={red!100!green!33},size=\small]{
\textbf{Comment [\uppercase{#1}\themycomment]:}~#2}
}}

\title[grids and divergence]{Badly approximable grids and $k$-divergent lattices}
\author{Nikolay Moshchevitin, Anurag Rao, Uri Shapira}
\address{Moscow Center for Pure and Applied Mathematics 
\newline
Institute for Information Transmission Problems,  Russia {\tt moshchevitin@gmail.com}}
\address{Peking University, China
{\tt anrg@bicmr.pku.edu.cn}}
\address{Technion -- Israel Institute for Technology {\tt ushapira@techniona.ac.il}}

\begin{document}

\begin{abstract}
    
    Let $A\in\on{Mat}_{m\times n}(\bR)$ be a matrix. In this paper we investigate the set 
    $\on{Bad}_A\subset \bT^m$ of badly approximable targets for $A$, where $\bT^m$ is the $m$-torus. It is well known that $\on{Bad}_A$ is a winning set for Schmidt's game and hence is a dense subset of full Hausdorff dimension. We investigate the relationship between the measure of $\on{Bad}_A$ and Diophantine properties of $A$.
    
    On the one hand, 
    we give the first examples of a non-singular $A$ such that $\on{Bad}_A$ has full measure with respect to some non-trivial algebraic measure on the torus. For this, we use 
transference theorems due to Jarnik and Khintchine,
and the parametric geometry of numbers in the sense of Roy.
    
    On the other hand, we give a novel Diophantine condition on $A$ that slightly strengthens non-singularity, and show that under the assumption that $A$
    satisfies this condition, $\on{Bad}_A$ is a null-set with respect to any non-trivial algebraic measure on the torus. For this we use naive homogeneous dynamics, harmonic analysis, and a novel concept we refer to as
    mixing convergence of measures.

\end{abstract}

\maketitle

\section{Introduction}

\subsection{The origins of this work}\label{sec: origins}
This paper stems from an effort to understand the state of affairs regarding the validity of several statements claimed to be true in~\cite[\S 2.2, \S 3.2]{GDV}, whose proof relied on a careless observation made by the third named author which is false, as was pointed out to us by David Simmons.

In this paper we revisit this discussion and prove some of the statements claimed in \cite{GDV} under slightly stronger assumptions. We also provide constructions of examples showing that some of the statements made in \cite{GDV} are false and maybe more importantly, we present some open problems.

In order to present the discussion in an organic manner and keep the reader 
motivated, we open with a self contained introduction to the subject. 
In \S\ref{sec: the mistake} we will elaborate further regarding what is exactly the mistake in \cite{GDV} and the current state of affairs regarding the questionable statements there.

\subsection{Inhomogeneous Diophantine approximation}
One of the main themes in the theory of inhomogeneous Diophantine approximations is to analyze, for a matrix 
$$A = \mat{\theta^1&\cdots &\theta^n}\in \on{Mat}_{m\times n}(\bR)$$
the rate at 
which the group generated by the columns of $A$ in $\bR^m$ approximates a given target vector $\eta\in \bR^m$ modulo the integers. 
More precisely, let $\norm{\cdot}$ denote choices of norms on $\bR^n$ and $\bR^m$ and let $\idist{\cdot}$ denote
the induced distance on the $m$-torus $\bT^m:=\bR^m/\bZ^m$. 
This theory tries to understand, for a given $\eta\in \bT^m$ the rate at which the sequence 
$$\min\seta{\idist{Aq -\eta}:q\in \bZ^n, \norm{q} \le Q}$$ approaches zero (if at all) as $Q\to\infty$. 
One way to do this is to fix a monotonely increasing function $\psi:\bR_{>0}\to\bR_{>0}$  and investigate
\begin{equation}\label{eq: liminf}
\liminf_{q\in \bZ^n, \norm{q}\to\infty} \psi(\norm{q})\idist{Aq -\eta}.
\end{equation}
 The most natural and widely investigated function $\psi$ is $\psi(t) = t^{n/m}$. A heuristic reason for why this is a natural choice for $\psi$ is that under the constraint $\norm{q}\le Q$ on the coefficient vector $q\in \bZ^n$, we have roughly $Q^n$ points in $\seta{Aq \mod \bZ:\norm{q}\le Q}\subset \bT^m$ and since the $m$-torus is $m$-dimensional, if we split it to $Q^n$ boxes of the same size, their side length is $Q^{-n/m}$. This is why the rescaling by $Q^{n/m}$ leads to an interesting discussion.

The most basic question one might ask about \eqref{eq: liminf} is whether it is positive or not. Indeed a classical subset of the torus which is widely investigated is the set of \textit{badly approximable targets for $A$}:  
\begin{equation}
    \label{eq: BadA}
    \on{Bad}_A := \seta{\eta\in \bT^m: \liminf_{q\in \bZ^n, \norm{q}\to\infty} \norm{q}^{n/m}\idist{Aq -\eta} >0}.
\end{equation}
In this paper we investigate the structure of $\on{Bad}_A$ and its relation to Diophantine conditions on $A$. The following classical result says that $\on{Bad}_A$ is never empty. 
\begin{thm}[Theorem X, Chapter IV \cite{CasselsDA}]
\label{thm: Bad non empty}
    For any $A \in \on{Mat}_{m\times n}(\bR)$, there exists an $\eta \in \R^m$ for which
    \begin{equation}\label{eq: liminf >0}
        \inf_{q \in \Z^n \smallsetminus \{0\}} \norm{q}^{n/m} \langle Aq - \eta \rangle >0.
    \end{equation}
\end{thm}
This result was amplified significantly as follows:
\begin{thm}[Theorem 1.4 in \cite{ETcrelle} or Theorem 1 in \cite{Moshchevitin-badforms}]\label{thm: Bad is winning}
    For any $A\in \on{Mat}_{m\times n}(\bR)$, the set $\on{Bad}_A\subset \bT^m$ is a winning set for Schmidt's game. As a consequence, it is a dense subset of full Hausdorff dimension.
\end{thm}
For the notion of winning and Schmidt games see the seminal paper \cite{Schmidt-games}.
See also \cite{BHKV} where the fact that $\on{Bad}_A$ has full Hausdorff dimension was first proved.
Theorem~\ref{thm: Bad is winning} says that $\on{Bad}_A$ is large from the dimension point of view. In this paper we are interested in investigating its size from the measure theoretical point of view. The basic question that 
we investigate is the following:
\begin{ques}\label{ques: main}
    Given a probability measure $\mu$ on $\bT^m$, what can be said about 
    $\mu(\on{Bad}_A)$ and how does this relate to the Diophantine properties of $A.$
\end{ques}
To examplify a situation where Question~\ref{ques: main} has an easy answer we note the following. Consider the case $n=1$ in which case $A=\theta$ is a single column vector (or rather a point in $\bT^m$). Assume its coordinates satisfy that $1,\theta_1,\dots,\theta_m$ are linearly dependent over $\bQ$. This is equivalent to the fact that the cyclic subgroup generated by $\theta$ in $\bT^m$ is not dense. In fact, in such a case, $\overline{\seta{q\theta:q\in \bZ}}$ is a union of finitely many 
cosets of a lower dimensional subtorus. In particular, any $\eta$ outside this lower dimensional submanifold belongs to $\on{Bad}_\theta$. In particular, we get that 
$$\mu(\on{Bad}_\theta)=1$$
for various measures $\mu$. For example, for $\mu$ being the Lebesgue measure
on $\bT^m$ as well as many \textit{algebraic measures} according to the following definition.
\begin{defn}
    A probability measure $\mu$ on $\bT^m$ is said to be algebraic if there is
    a subspace $U<\bR^m$ whose image in $\bT^m$ is closed (i.e. $U\cap \bZ^m$ is a lattice in $U$), and $\mu$ is the $U$-invariant probability measure supported on a single $U$-orbit in $\bT^m$. We exclude the possibility of 
    $U=\seta{0}$ by saying that $\mu$ is non-trivial.
\end{defn}
The Diophantine condition saying that $1,\theta_1,\dots, \theta_m$ are linearly dependent over $\bQ$ is very strong. A relaxation of it is the famous notion of \textit{singularity}:

\begin{defn}[Singularity]
  We say that $A \in \on{Mat}_{m\times n}(\bR)$ is  singular if, for every $\varepsilon >0$, for all large enough $Q$ 
  \begin{equation}\label{eq: singular defn}
    Q^{n/m}\min\seta{\idist{Aq}: q\in \bZ^n\sm\seta{0}, \norm{q}\le Q}<\varepsilon.
\end{equation}
\end{defn}
Assuming non-singularity of $A$ we have the following result from 
\cite{GDV} that answers Question~\ref{ques: main} for the Lebesgue measure on
$\bT^m$.
\begin{thm}[\cite{GDV}]\label{thm: Uri lebesgue}
    Let $A\in \on{Mat}_{m\times n}(\bR)$ be non-singular and let $\lam_{\bT^m}$ denote the Lebesgue probability measure on $\bT^m$. Then $\lam_{\bT^m}(\on{Bad}_A) = 0$.
\end{thm}
As pointed out in \S\ref{sec: origins}, there is a mistake in the paper \cite{GDV}. Theorem~\ref{thm: Uri lebesgue} is a special case of \cite[Theorem 2.3]{GDV}, which in the generality stated there turns out to be false as we will see below. 
But, the argument presented there, for the case of the measure $\lam_{\bT^m}$, is robust enough to carry through. We will provide a full proof of Theorem~\ref{thm: Uri lebesgue} reproducing the argument of \cite{GDV} in \S\ref{sec: proof for Lebesuge}. 
This proof will also serve as 
an introduction to the proof of one of our main results Theorem~\ref{thm: main theorem F formulation} in \S\ref{sec: long sequences}. 
We note that since \cite{GDV}, several alternative proofs for Theorem~\ref{thm: Uri lebesgue} have appeared. See \cite[Corollary 1.4]{Kim-Kurzweil}, \cite[Theorem 1]{Mosh-welldistributed} and \cite[Theorem 1.3]{BDGW} and also \cite[Theorem 1.2]{DHKim-Nonlinearity} for the one-dimensional case.
See also Remark~\ref{rem: Kolya} and Theorem~\ref{thm: Kolya} for a slightly stronger result than Theorem~\ref{thm: Uri lebesgue} which we 
prove in \S\ref{sec: Kolya}.

Since David Simmons spotted the gap in \cite[\S2.2,\S3.2]{GDV} the third named author tried to rule regarding the validity/falsity of the results there, most notably regarding \cite[Theorem 2.3]{GDV}, which implies that if 
$A$ is non-singular, then $\mu(\on{Bad}_A)=0$ for every non-trivial 
algebraic 
measure $\mu$ on $\bT^m$. One of the main results we present in this paper is the following construction which shows that this statement may fail drastically. In it we choose $n=1$ and take $A=\theta$ to be a vector.
\begin{thm}\label{thm: counterexample intro}
    Let $m \in \N$ with $m > 2$. There exists $\theta \in \bR^m$ which is non-singular and an $\eta \in \bR^m$ such that, for every $(t_3,
    \dots, t_m) \in \R^{m-2}$, we have
    \begin{equation*}
        \inf_{q \in \Z \smallsetminus \{0\}} \ |q|^{1/m} \idist{q\theta  -( \eta + t_3\ve_3 + \dots + t_m \ve_m)} > 0.
    \end{equation*}
    In particular $\on{Bad}_\theta$ contains a coset of a subtorus of codimension $2$ and as a consequence, there are non-trivial algebraic measures $\mu$ satisfying $\mu(\on{Bad}_\theta) =1$.
\end{thm}
This theorem is proved as Theorem \ref{thm: counterexample} below.
    Here $\ve_3,\dots, \ve_m$ denote last $m-2$ standard basis vectors for $\R^m$.
\begin{rem}
    We note here that it is very likely that Theorem~\ref{thm: counterexample intro} could be generalized to $n>1$. This would probably involve using the new 
    theory of parametric geometry of numbers as in \cite{DFSU}. We were content with using this theory as presented in \cite{Roy-MathZ} which seems to give only the $n=1$ case.
\end{rem}
\begin{rem}
    We note here that we were unable to provide a construction of non-singular
    $\theta\in \bR^m$ such that $\mu(\on{Bad}_\theta)>0$ for an algebraic measure corresponding to a co-dimension 1 subtorus. In particular, the case $m=2$ remains open. That is, is there a non-singular vector $\theta\in \bR^2$ and an algebraic measure $\mu$ on $\bT^2$ of dimension 1 such that $\mu(\on{Bad}_\theta)>0$?
\end{rem}
We proceed to state versions of our main result which could be viewed as an amendment of \cite[Theorem 2.3]{GDV}. We introduce a novel Diophantine condition on $A$ that strengthens non-singularity and ensures that $\mu(\on{Bad}_A)=0$ for any non-trivial algebraic measure $\mu$ on $\bT^m$. 
This Diophantine condition is a bit elaborate to state and is dynamical in nature. 

Let $d:=n+m$. Let $X$ denote the space of unimodular (i.e.\ covolume 1) lattices 
in $\bR^d$. As usual $X$ is identified with the quotient $\SL_d(\bR)/\SL_d(\bZ)$ via $g\SL_d(\bZ)\leftrightarrow g\bZ^d$. This identification gives rise to a topology on $X$ (the quotient topology), as well as to a natural action of $\SL_d(\bR)$ on $X$. 
Given $A \in \on{Mat}_{m\times n}(\R)$ we consider the lattice
\begin{equation}\label{eq: xA}
    x_A := \left[ {\begin{array}{cc} I_m & A \\ 
    0 & I_n  \end{array}}\right]\Z^d.
\end{equation}
Of particular interest to our discussion is the following one-parameter subgroup
\begin{equation}\label{eq: ht definition}
    h_t := \left[ {\begin{array}{cc} e^{nt}I_m & 0 \\ 
    0 & e^{-mt}I_n  \end{array}}\right] \in \SL_d(\bR),\;\; t\in \bR.
\end{equation}
The so called \textit{Dani correspondence} (see \cite{Dadivergent}) says that many 
of the Diophantine properties of $A$ could be read from the topological and statistical properties of the orbit $\seta{h_tx_A:t\ge 0}$. 
In particular, we have the following famous characterization of singularity.
\begin{thm}[Theorem 2.14 in \cite{Dadivergent}]\label{thm: Dani singular}
    A matrix $A \in \on{Mat}_{m\times n}(\R)$ is singular if and only if the orbit 
    $\seta{h_tx_A:t\ge 0}$
    is divergent.
\end{thm}
Here the orbit is said to be divergent if the map $t\mapsto h_tx_A$ is a \textit{proper} map from $\bR_{\ge0}$ to $X$ (i.e. preimages of compact sets are compact). 
Our novel Diophantine condition on $A$ is a relaxation of singularity
that is stated in terms of a topological property of the orbit $\seta{h_tx_A:t\ge 0}$. 
\begin{defn}[Asymptotic accumulation points]\label{defn: asymptotic limit points}
    Let $x \in X$. We define the set of asymptotic accumulation points of $x$ as
    \begin{equation*}
        \partial(x) := \left\lbrace y \in X: \text{there is an unbounded sequence } (t_k)_{k \in \N} \subset \R_{\geq 0} \text{ with } \lim_{k \to \infty} h_{t_k}x = y \right\rbrace.
    \end{equation*}
\end{defn}
\begin{defn}[Accumulation sequences and $k$-divergence]\label{defn: kdiv lattice}
    Let $x \in X$. A sequence
    \begin{equation*}
        I = (x_0=x, x_1, \dots , x_k) \subset X^{k+1}\ \text{ with }\ x_{i+1} \in \partial(x_i) \text{ for } i=0,\dots k-1
    \end{equation*}
    is called an accumulation sequence of length $k+1$ for $x$.
    If 
    \begin{equation*}
        k= \sup \left\lbrace \on{length}(I) - 1 : I \text{ is an accumulation sequence for } x\right\rbrace,
    \end{equation*}
    we say $x$ is $k$-divergent. A matrix $A$ is said to be $k$-divergent if the lattice $x_A$ is.
\end{defn}
\begin{rem}
    The set $\partial(x)$ and as a consequence the definition of $k$-divergent lattices depend on the semi-group $\seta{h_t:t\ge0}$. Versions 
    of these notions could be investigated for other groups and semigroups. We choose to work with this particular semi-group because of the relation to Diophantine approximation given by the Dani correspondence as reflected in Theorem~\ref{thm: Dani singular}.
\end{rem}
Note that by Theorem~\ref{thm: Dani singular}, saying that $A$ is singular is the same as saying it is $0$-divergent. In the following results we assume $A$ is not $k$-divergent for small values of $k$. This gives enough dynamical richness for the orbit $\seta{h_tx_A:t\ge 0}$ for our arguments to carry through.  
\begin{thm}\label{thm: 1 dim in Zm-2}
    Assume  $A\in \on{Mat}_{m\times n}(\bR)$ is not $k$-divergent for
    any $0\le k\leq d-2$ (in other words, the lattice $x_A$ has an accumulation sequence of length $d$), and that $\gcd(m,n)=1$. Then, for any non-trivial algebraic measure $\mu$ on $\bT^m$ we have that 
    $$\mu(\on{Bad}_A) = 0.$$
\end{thm}
\begin{rem}
    The assumption $\gcd(n,m)=1$ appearing in some of our results is curious.  We are not sure how much of it is an artifact of the method of proof.
    See Lemma~\ref{lemma: zero is not a weight} and Corollary~\ref{cor: subspaces do not stabilize} where this assumption enters the discussion. This assumption is not entirely redundant though, as 
    the case $n=m$ is generally false. See Example~\ref{exmp: counter n=m}.
\end{rem}
Theorem \ref{thm: 1 dim in Zm-2} could be restated more explicitly as follows. Since any $\ell$-dimensional subtorus of $\bT^m$ can be presented as a product of a 1-dimensional subtorus and a $\ell-1$-dimensional one, by Fubini's theorem we see that the general statement of Theorem~\ref{thm: 1 dim in Zm-2} 
follows from the corresponding statement for 1-dimensional subtori. Since 1-dimensional subtori of $\bT^m$ are in 1-2 correspondence to primitive integral vectors (up to sign), the following is a restatement of Theorem
\ref{thm: 1 dim in Zm-2}.
\begin{thm}\label{thm: 1 dim in Zm}
Assume $A\in \on{Mat}_{m\times n}(\bR)$ is not $k$-divergent for any $0\le k\le 
d-2$ (in other words, the lattice $x_A$ has an accumulation sequence of length $d$), and $\gcd(n,m)=1$. Then,    
for any $p_0 \in \Z^m$ and $\eta \in \R^m$, we have for Lebesgue almost every $t \in \R$
    \begin{equation*}
        \liminf_{q\in \bZ^n, \norm{q}\to\infty} \norm{q}^{n/m} \langle Aq - (tp_0 + \eta)\rangle = 0.
    \end{equation*}
\end{thm}

Theorems~\ref{thm: 1 dim in Zm-2}, \ref{thm: 1 dim in Zm} give the answer to Question~\ref{ques: main} for the case when $\mu$ is algebraic. Our techniques give the following stronger result which deals with the one-dimensional Lebesgue measure sitting on an immersed line in the direction of $Aq_0+p_0$ for any given $q_0\in \bZ^n, p_0\in \bZ^m.$ 
These immersed lines are closed if $q_0 = 0$. 
\begin{thm}\label{thm: 1 dim in theta}
    Assume $A\in \on{Mat}_{m\times n}(\bR)$ is not $k$-divergent for any 
    $0\le k\le d-2$ (in other words, the lattice $x_A$ has an accumulation sequence of length $d$), and that $\gcd(n,m)=1$. Then, for any $p_0\in \bZ^m, q_0\in \bZ^n$ and $\eta \in \R^m$, we have for Lebesgue almost every $t \in \R$
    \begin{equation*}
        \liminf_{q\in \bZ^n, \norm{q}\to\infty} \norm{q}^{n/m} \langle Aq - (t(Aq_0+p_0) +  \eta)\rangle = 0.
    \end{equation*}
\end{thm}
Theorem \ref{thm: 1 dim in theta} reduces to  Theorem \ref{thm: 1 dim in Zm} if we choose $q_0=0$. Theorem~\ref{thm: 1 dim in theta} follows from the more general Theorem \ref{thm: main theorem F formulation} below as we show after its formulation in \S\ref{sec: VF and inheritance}.
\begin{rem}
    Note that when $m=1$ and $n\ge 2$, Theorems~\ref{thm: 1 dim in Zm}, \ref{thm: 1 dim in Zm-2}, \ref{thm: 1 dim in theta} are weaker than Theorem~\ref{thm: Uri lebesgue} because there is only one non-trivial 
    algebraic measure on $\bT^1=\bT^m$. So in this case there is no need to assume $x_A$ is not $k$-divergent for $0\le k\le d-2$. It is enough to assume 
    it is not $k$-divergent for $k=0$. 
\end{rem}

\begin{rem}
    It is non-trivial to construct $A$ which is $k$-divergent. Using 
    the emerging theory of \textit{parametric geometry of numbers}, $k$-divergent lattices are constructed in \cite{LRSY} and in fact, the Hausdorff dimension of the set of $k$-divergent lattices is calculated there. Curiously enough, the vector $\theta \in \R^m$ constructed in Theorem \ref{thm: counterexample intro} must be $k$-divergent for some $k=1,\dots, d-2$.
\end{rem}

\begin{rem}
This remark pertains to a possible variant on how to define $\on{Bad}_A$.
In equation~\eqref{eq: BadA} we took a liminf as $q\in \bZ^n, \norm{q}\to\infty$. It is sometimes desirable to restrict attention to coefficient vectors $q$ which belong to some subset. The most natural situation in which this arises is when $n=1$, i.e.\ $q\in \bZ$ but we are interested in $q\to \infty$ rather than $\av{q}\to \infty$. In many discussions in Diophantine approximations it is trivial to connect two such questions but in our inhomogeneous setting, 
it \textbf{does not} seem to be true that for a given vector $\theta\in \bR^m$ and a target $\eta\in \bR^m$  
$$\liminf_{q\to \infty}q^{1/m}\idist{q\theta-\eta}=0\Longleftrightarrow \liminf_{\av{q}\to\infty}\av{q}^{1/m}\idist{q\theta-\eta} = 0.$$
If one defines a variant $\on{Bad}_A^+$ of $\on{Bad}_A$ using the restriction that 
$\norm{q}\to\infty$ and all the coordinates of $q$ are positive, then we expect that the techniques of this paper should be strong enough to prove versions of Theorems \ref{thm: Uri lebesgue} and \ref{thm: 1 dim in Zm-2}. Such statements are stronger because $\on{Bad}_A\subset \on{Bad}_A^+$. 
Here is an outline of how one might do this: One defines a version of the value set $V_F(y)$ (see Definition \ref{def: value set}) that 
restricts attention to the values $F$ takes on the grid points in the cone
in $\bR^d$ corresponding to the last $n$-coordinates being positive. One then proves a version of Lemma~\ref{lem: non degeneracy degree}, where instead of the non-degeneracy degree defined in Definition~\ref{defn: non degeneracy degree}, one considers only grid points in the cone. The rest of the argument then follows the same path as in the present paper. We note that there are results in the literature regarding homogeneous Diophantine approximations
with sign constraints on the coefficients. See \cite{SchmidtPositiveCoef, Thurnheer1, Thurnheer2, MoshchevitinPositive, RoyPositive}.
\end{rem}
\begin{rem}\label{rem: Kolya}
As we saw in Theorem~\ref{thm: Uri lebesgue}, if $A$ is non-singular then $\lam(\on{Bad}_A)=0$. A natural question that comes to mind is if
the opposite statement is true, namely is it true that if $A$ is
singular, then $\lam(\on{Bad}_A)>0$. As we shall now explain, this is not the case. In fact, we can give an explicit Diophantine condition 
on $A$, which gives rise to a class of matrices strictly containing the non-singular matrices (and in particular, this class contains some singular matrices), such that for $A$ in this class, $\lam(\on{Bad}_A)=0$. In particular, this shows that there are singular matrices $A$ with $\lam(\on{Bad}_A)=0$. The definition of the Diophantine class of matrices is stated in terms of the sequence of best
approximations and is tightly related and motivated by the discussions in the papers \cite{Mosh-welldistributed} and \cite{Kim-Kurzweil}. We note that in \cite{Kim-Kurzweil}, another complementary 
Diophantine class is defined, 
which is a subclass of the singular matrices, and it is proved there that for such matrices $A$ one has $\lam(\on{Bad}_A)=1$. 

In order to define our new Diophantine condition properly, one needs a bit of notation and terminology and hence we postpone the exact formulation of the result to \S\ref{sec: Kolya}. See Definition~\ref{defn: Kolya} and Theorem~\ref{thm: Kolya}.
\end{rem}
\subsection{A few open problems}
\ignore{
\subsubsection*{Characterizing singularity}
It is not clear to us if one can characterize the singularity of a 
matrix $A$ by the measure of $\on{Bad}_A$ with respect to various measures. 
We know from Theorem~\ref{thm: Uri lebesgue} that if $A$ is non-singular, then
$\lam_{\bT^m}(\on{Bad}_A)=0$. In \cite[Theorem 1.7]{Kim-Kurzweil} the author gives a non-trivial explicit 
Diophantine subclass of singular matrices  that satisfy 
$\lam_{\bT^m}(\on{Bad}_A)=1$. It is interesting to ask whether or not there exists singular matrices $A$ with 
$\lam_{\bT^m}(\on{Bad}_A)=0$. It seems to us that one might be able to provide a positive answer to this question using the techniques in \cite{Mosh-welldistributed}.
}
\subsubsection*{Non trivial coset intersection} In Theorem \ref{thm: counterexample intro}, the co-dimension 2 coset is completely contained in 
$\on{Bad}_A$. Example \ref{exmp: counter n=m} also shows a similar phenomenon. The most basic question one might state here is the following: Are there examples of non-singular matrices $A$ and one dimensional algebraic measures $\mu$ on $\bT^m$ such that $\mu(\on{Bad}_A)>0$ and the support of 
$\mu$ not fully contained in $\on{Bad}_A$. 

\subsubsection*{Co-dimension 1 algebraic measures}
Can one construct a non-singular matrix $A$ 
such that $\mu(\on{Bad}_A)>0$ for an algebraic measure $\mu$ of dimension 
$d-1$?

\subsubsection*{The co-primality condition}
It is not clear to us at the moment if the condition $\gcd(m,n)=1$ in
Theorem~\ref{thm: 1 dim in Zm-2} can be relaxed. Example \ref{exmp: counter n=m} suggests that the case $n=m$ is false. 

\subsubsection*{$k$-divergence}
Our definition of $k$-divergence for a matrix $A$ is dynamical 
but we expect that similarly to 
singularity, there should be a purely Diophantine characterization of this property, although we were not able to pin down such a condition. 
Even in the case $n=m=1$, where $A$ is a number, and $0$-divergence is characterized by rationality, we do not have a clean characterization of the notion of $k$-divergence (for general $k$) using continued fractions.  

A question that emerges from Theorem~\ref{thm: counterexample intro} is the following: In Theorem~\ref{thm: counterexample intro} we construct a non-singular vector that violates the conclusion of Theorem~\ref{thm: 1 dim in Zm-2} and hence it must be $k$-divergent for some $1\le k\le d-2$. For which values of $k$ can this be achieved? 
A possible reformulation of this question might be: For a given $1\le k\le d-2$, is it possible to construct 
$k$-divergent matrices $A$ and algebraic measures $\mu$ on $\bT^m$ such that 
$\mu(\on{Bad}_A)>0$?

\subsubsection*{Beyond algebraic measures} 
In the context of Theorems~\ref{thm: Uri lebesgue}, \ref{thm: 1 dim in Zm-2}, 
it is natural to ask what happens for other natural classes of measures. Given a natural class $\cA$ of measures on $\bT^m$, can one pin down exact Diophantine conditions on $A$ which ensure that for all $\mu\in \cA$, 
$\mu(\on{Bad}_A)=0$? Examples of $\cA$ can be, smooth measures on submanifolds, Hausdorff measures on certain fractals, or measures whose Fourier coefficients satisfy certain conditions.

\subsubsection*{Zero-One Law} Is it always the case that $\mu(\on{Bad}_A)\in \seta{0,1}$ for algebraic measures on $\bT^m.$ For what classes of measures 
does such a zero-one law hold? The following simple observation gives 
an affirmative answer in the case that $\mu$ is an algebraic measure supported on a subgroup (rather than a coset):
\begin{prop}
    Let $\mu$ be an algebraic measure supported on a subgroup of 
    $\bT^m$. Then for any $A\in \on{Mat}_{m\times n}(\bR)$ we have 
    that $\mu(\on{Bad}_A)\in \seta{0,1}.$
\end{prop}
\begin{proof}
    Let $T:\bT^m\to \bT^m$ denote the homomorphism $\eta\mapsto T(\eta):=2\eta.$ Then, as $\mu$ is the Haar measure on a subgroup, $\mu$ is $T$-invariant and ergodic (see \cite[Corollary 2.20]{EW}).
    We will therefore finish the proof once we show that $T^{-1}(\on{Bad}_A)\subset \on{Bad}_A$. This is equivalent to showing that  
    $$\eta\notin \on{Bad}_A \implies 2\eta\notin\on{Bad}_A.$$ Indeed, if $\eta\notin\on{Bad}_A$, then by definition
    $$\liminf_{q\in \bZ^n,\norm{q}\to\infty}\norm{q}^{n/m}\idist{Aq-2\eta}=0$$
    and so necessarily 
    \begin{align*}
    \liminf_{q\in \bZ^n,\norm{q}\to\infty}\norm{q}^{n/m}
    \idist{Aq-2\eta}&\le
    \liminf_{q\in \bZ^n,\norm{q}\to\infty}\norm{2q}^{n/m}\idist{A2q-2\eta} \\
    &\le 2^{n/m+1}\liminf_{q\in \bZ^n,\norm{q}\to\infty}\norm{q}^{n/m}\idist{Aq-\eta}=0 
    \end{align*}
    and so $2\eta\notin \on{Bad}_A$.
\end{proof}

\section{From Inhomogeneous approximations to value sets of grids}
In this section we introduce most of the notation and terminology used in the paper, state the main Theorem~\ref{thm: main theorem F formulation}, and prove some basic results like the Inheritance Lemma~\ref{lem: inheritance}.

\subsection{Euclidean space}
Throughout the paper we fix integers $m,n\ge 1$ and set 
$$d:=m+n.$$ 
Vectors in $\bR^d$ will be denoted $u,v,w$ etc. The decomposition of a vector in $\bR^d$ to its first $m$ coordinates and last $n$ coordinates will be important from time to time. Hence, when we write a vector as $ \smallmat{v\\w}$ the reader should interpret $v\in \bR^m, w\in \bR^n$.
Given two subsets $A,B\subset \bR^d$, we denote $A+B = \seta{v+w:v\in A, w\in B}$.

\subsection{The space of lattices}
We let 
$X$ denote the space of unimodular lattices in $\bR^d$. Although it can be 
confusing, we denote lattices by small letters $x, x_1, x_2$ etc. It will be important to us to think of $X$ as a topological space, and thus be able to discuss, for example, a converging sequence $(x_l)_{l\in \bN}$ in $X$. At the 
same time, we wish to keep in mind that a lattice $x\in X$ is also viewed simply a subset of $\bR^d$. 

A simple way to understand the topology on $X$ is to note that the group
$G:=\SL_d(\bR)$ acts transitively on $X$, where the action $(g,x)\mapsto gx$ 
is induced by the linear action of $G$ on $\bR^d$. That is, given a lattice $x\in X$ and $g\in G$, the lattice $gx$ is simply the linear image of $x$ under the linear map induced by multiplication by $g$ on column vectors. 

It is easy to see that this action is transitive and that the stabilizer group of the standard lattice $\bZ^d$ is $\Ga:= \SL_d(\bZ)$. We therefore obtain bijection $g\Ga\leftrightarrow g\bZ^d$ between the formal coset space $G/\Ga$ and the space $X$. An immediate advantage of this is that the quotient
$G/\Ga$ is naturally equipped with a topology -- the quotient topology with respect to the natural projection map $G\to G/\Ga$. In fact, since $\Ga<G$ is discrete, this map is a covering map. In particular, we have the following simple result which allows us to explicitly understand convergence in $X$ and 
will be used without reference throughout.
\begin{lem}
    A sequence $(x_l)_{l\in \bN}$ in $X$ converges to a lattice $x\in X$ if and only if there exists a sequence $\varepsilon_l\in G, \varepsilon_l\to e$ ($e$ being the identity element), such that for all $l$, 
    $\varepsilon_l x_l=x.$
\end{lem}

\subsection{The space of grids}
Most of our discussion will take place in a topological space $X\subset Y$. 
Again, we first introduce $Y$ as an abstract set and then explain how to 
equip it with a topology. 
\begin{defn}
    \label{defn: grids}
    A unimodular grid in $\bR^d$ is a subset of the form 
    $$x+v: = \seta{u+v:u\in x}$$
    where $x\in X$ and $v\in \bR^d$.
\end{defn}
The space of all unimodular grids will be denoted by $Y$. Note that in the representation of a grid $y\in Y$ as $y=x+v$, the lattice $x$ is uniquely determined (it is obtained by translating $y$ so that it would contain $0$)
but the translation vector $v$ is only well defined modulo $x$. 
Note also that $X\subset Y$ -- any lattice is a grid and a grid $y$ is a lattice if and only if $0\in y$.

In order to induce a topology on $Y$ we follow the same line of thought as before. Let 
$$G':=\on{ASL}_d(\bR) := \seta{\mat{g&v\\ 0&1}\in \SL_{d+1}(\bR):g\in \SL_d(\bR), v\in \bR^d}$$
and consider the action of $G'$ on $\bR^d$ given by 
$$\mat{g&v\\ 0&1}\cdot u = gu+v.$$
Thus the matrix $\smallmat{g&v\\0&1}\in G'$ acts on $\bR^d$ as an affine map obtained by applying the linear map $g$ followed by the translation by $v$.

The action of $G'$ on $\bR^d$ induces an action on subsets of it and as a consequence $G'$ acts on $Y$.
It is easy to see that this action is transitive and that the stabilizer group
of the standard grid $\bZ^d$ is $\Ga'=\on{ASL}_d(\bZ)$ (i.e.\ the matrices in 
$G'$ whose entries are integral). We therefore identify $Y$ with $G'/\Ga'$ 
via $\smallmat{g&v\\0&1}\Ga'\leftrightarrow g\bZ^d+v$ and equip $Y$ with the 
quotient topology with respect to the natural quotient map $G'\to G'/\Ga'$.
The following Lemma is left to be verified by the reader.
\begin{lem}
    \label{lem: vectors converge}
    Let $y, y_l\in Y, l\in \bN$ be grids and assume that $y_l\to_{l\to\infty} y$. Then for any vector $v\in y$, there exist vectors $u_l\in y_l$ such that $u_l\to v$. 
\end{lem}
We note that although the transitive action of $G'$ on $Y$ is used 
to define the topology, this action will not play any other role in this 
paper. But, the group $G$ acts on subsets of $\bR^d$ and thus acts on $Y$ in a way that extends its action on $X$. The $G$-action on $Y$ will be paramount to our discussion. 
\subsection{The projection from grids to lattices and its fibers}
Let $\pi:Y\to X$ denote the map that sends a grid $y$ to the unique lattice $x$ such that $y=x+v$ for some $v\in \bR^d$. 
We leave it to the enthusiastic reader to check that $\pi$ is a continuous
proper map, where proper means that preimages of compact sets are compact.

Given a lattice $x\in X$, the fiber $\pi^{-1}(x)=\seta{x+v:v\in \bR^d}$ is nothing but the collection of all cosets of the lattice $x$ in $\bR^d$. As such it is \textbf{equal} to the torus $\bR^d/x$. The space $Y$ can be thus thought of as the union of all possible tori (of volume 1) of $\bR^d$. 
We alternate between writing $\pi^{-1}(x)$ and $\bR^d/x$ throughout but wish to stress here that we can (and will) alternate between thinking of a tori
$\bR^d/x$ as
\begin{itemize}
    \item The compact abelian group $\bR^d/x$.
    \item A closed subset of $Y$ obtained as a fiber $\pi^{-1}(x)$.
    \item A collection of grids (subsets of $\bR^d$) $\seta{x+v:v\in \bR^d}$.
\end{itemize}
Finally note that $\pi$ intertwines 
the $G$-actions on $X, Y$. That is, the following diagram commutes: For any $g\in G$,
$$ \xymatrix{
Y\ar[r]^g\ar[d]_\pi & Y\ar[d]^\pi \\
X\ar[r]^g &X
}$$
This captures the fact that if we are given an element 
$g\in G$ and two lattices $x_1,x_2\in X$ such that $gx_1 = x_2$, then when we act with $g$ on $Y$, $g$ maps the fiber $\pi^{-1}(x_1)$ onto the fiber $\pi^{-1}(x_2)$. Indeed the map $x_1+v\mapsto g(x_1+v)$ is an isomorphism between the compact abelian groups $\bR^d/x_1, \bR^d/x_2$. 
\subsection{Subtori, Haar measures, and algebraic measures}
As noted above, for a lattice $x\in X$, the fiber $\pi^{-1}(x)=\bR^d/x$ is 
a $d$-dimensional compact abelian group we refer to as a torus.
We now discuss its closed connected subgroups, their Haar probability measures and their translates.

Let $x\in X$ be a lattice. A linear subspace $U<\bR^d$ is called 
\textit{$x$-rational} if $U\cap x$ contains a basis of $U$.
It is well known that there is a 1-1 correspondence between $x$-rational subspaces and closed connected subgroups of $\bR^d/x$. Let $U<\bR^d$ be a 
$x$-rational subspace. 
For any grid
$y\in \bR^d/x$, the set $y+U\subset \bR^d/x$ is a coset of the closed subgroup $x+U<\bR^d/x$
and it supports a unique $U$-invariant probability measure. We refer to such measures as \textit{algebraic measures} on $\bR^d/x$. When the lattice $x$ (the trivial element of the torus) belongs to the support of the measure, we refer to the measure as a \textit{Haar measure} on a subtorus. 

\subsection{Value sets and inheritance}
\label{sec: VF and inheritance}
Let $F:\bR^d=\bR^m\oplus \bR^n\to \bR$ be the map
$$F(\mat{v\\w}) = \norm{v}^n\cdot\norm{w}^m.$$ 
Recalling the one-parameter subgroup $\seta{h_t:t\in \bR}<G$ defined in
\eqref{eq: ht definition}, we note that $F$ is $h_t$-invariant. 
That is, for any $t\in \bR$ and any $u\in \bR^d$, $F(h_tu)=F(u)$. Geometrically, when $h_t$ acts linearly on $\bR^d$, it acts on the level sets of $F$. Note that although $h_t$ does not act transitively on the level sets, it does act cocompactly on each level set of the form $F^{-1}(s)$ for $s>0$. This will not be used explicitly in the paper but has conceptual importance 
towards the claim that understanding the values $F$ takes on a subset of $\bR^d$ can be attacked by analyzing how this set changes under the action of $h_t$. This is the idea behind our results.  
\begin{defn}\label{def: value set}
Given a grid $y\in Y$ we define the \textit{value set} of $y$ to be
$$V_F(y)=\seta{F(u):u\in y}.$$
\end{defn}
The fundamental question in \textit{geometry of numbers} which guides us
is: What can be said about the value sets $V_F(y)$. Is it dense or discrete? Is $0$ an accumulation point or not? Does its closure contain a ray? 
For the sake of the discussion in this paper we make the following definitions.  
\begin{defn}\label{defn: DVF}
    A grid $y\in Y$ is a \textit{dense values grid} or $y$ is $DV_F$ 
    if $$\overline{V_F(y)}= F(\bR^d) = [0,\infty).$$
\end{defn}
The following lemma is a fundamental tool in our discussion. It serves 
as the entry point of dynamics to our discussion.
\begin{lem}[Inheritance lemma]\label{lem: inheritance}
Let $y_1,y_2\in Y$ be grids and assume  the orbit-closure $\overline{\seta{h_ty_1:t\in \bR}}$ in $Y$ contains $y_2$. Then, 
$\overline{V_F(y_2)}\subset \overline{V_F(y_1)}.$  
\ignore{
Also, if $\overline{\seta{h_ty_1:t\in \bR}}$ contains a coset of a $d-1$-dimensional 
subtorus in a fiber, then $y_1$ is $DV_F$.
}
\end{lem}
\begin{proof}
    Let $v\in y_2$. We show that $F(v)$ can be approximated by elements from 
    $V_F(y_1)$. 
    Let $t_l, l\in \bN$ be a sequence of real numbers such that $h_{t_l}y_1\to y_2$. By Lemma~\ref{lem: vectors converge}, there are vectors $u_l\in h_{t_l}y_1 $ such that $u_l\to v$. The vector $w_l:=h_{t_l}^{-1}u_l$ belongs to the grid $y_1$ and by the $h_t$-invariance of $F$ we have 
    together with the continuity of $F$, we have
    $$F(w_l) = F(h_{t_l}(w_l)=F(u_l)\to_{l\to\infty} F(v).$$
    This shows that $F(v)\in \overline{V_F(y_1)}$ and finishes the proof. 
\end{proof}
The way in which the Inheritance Lemma~\ref{lem: inheritance} will enter our proofs and help us establish that certain grids are $DV_F$ is via the following:
\begin{prop}\label{prop: codimension 1 subtorus}
    Let $y\in Y$ be a grid. If there exists a lattice $x\in X$ such that 
    the orbit-closure $\overline{\seta{h_ty:t\in\bR}}$ in $Y$ contains
    a full coset of a $d-1$-dimensional subtorus in the fiber 
    $\pi^{-1}(x) = \bR^d/x$, then $y$ is $DV_F$.
\end{prop}
For the proof, we will need the following Definition and Lemma.
\begin{defn}\label{defn: non degeneracy degree}
    The non-degeneracy degree of $F$ is the minimal dimension $\ell$ such that for any grid $y$ and any $\ell$-dimensional subspace $U<\bR^d$ one has $F(y+U)=[0,\infty).$
\end{defn}
\begin{lem}\label{lem: non degeneracy degree}
    The non-degeneracy degree of $F$ is $d-1.$ 
\end{lem}
\begin{proof}
    Let 
    $$y=\bZ^d + 1/2\cdot\mat{1\\ \vdots \\ 1}.$$
    Let $U<\bR^d$ be the co-dimension 2 subspace given by $U=\seta{\smallmat{v\\w}:v_1=0, w_n=0}$. Note that the first and last coordinates of the vectors in $y+U$
    are at least $1/2$ in absolute value. Therefore their $F$-value is $\ge 2^{-n}\cdot 2^{-m} = 2^{-d}.$ This shows that the non-degeneracy degree is at least $d-1$. 

    On the other hand, let $U<\bR^d$ be a $d-1$-dimensional space  let $y$ be a grid. We will show that $F(y+U)=[0,\infty).$ 
    Let $\smallmat{v^{1}\\w^{1}},\dots, \smallmat{v^{d-1}\\w^{d-1}}$ be a basis for $U$. Let 
    $$r_1 = \on{rank}\smallmat{v^1&\cdots& v^{d-1}},\; r_2 = 
    \on{rank}\smallmat{w^1&\cdots& w^{d-1}}.$$
    Since the rank of the matrix whose columns are the basis of $U$ is $d-1=m+n-1$, it is impossible to have both $r_1< m$ and $r_2< n$. Assume for concreteness that $r_1=m$. Choose a vector $\smallmat{p\\q}\in y$ such that $F(\smallmat{p\\q})>0$ (this is always possible because grids always contain points in $\bR^d_{>0}$ for example). We can solve the system of equations
    $$-p = a_1v^1+\dots +a_{d-1}v^{d-1}$$
    because we assume $r_1=m$. Therefore, the affine subspace $\smallmat{p\\q}+U$ contains a vector of the form $\smallmat{0\\w}$. 

    Assume for the moment that the norms defining $F$ are the Euclidean norms. The restricted function $F|_{\smallmat{p\\q}+U}$ is the square root of a polynomial in $d-1$ variables that attains the value 0 because of the above argument and also attains the value $F(\smallmat{p\\q})>0$. Hence it is a non-constant polynomial and so we conclude that $F(\smallmat{p\\q}+U) = [0,\infty)$ which finishes the proof. 

    We leave it to the reader to deduce the general case, where the norms used to define $F$ are not assumed to be the Euclidean norms, from the Euclidean one (use the equivalence of norms and the mean-value theorem).

\end{proof}
\begin{proof}[Proof of Proposition~\ref{prop: codimension 1 subtorus}]
    A subtorus of dimension $d-1$ in $\pi^{-1}(x)=\bR^d/x$ 
    is a subgroup of the form $x+U$ where $U<\bR^d$ is a $d-1$-dimensional 
    subspace whose image in $\bR^d/x$ is closed. A coset of such a subtorus, is a set of the form $x+v +U$ for such a subspace $U$ and a vector $v\in \bR^d$. If the closure $\overline{\seta{h_ty:t\in \bR}}$ contains 
    all the grids in such a set, then by the Inheritance 
    Lemma~\ref{lem: inheritance} we have that 
    $$\overline{V_F(y)}\supset \bigcup_{u\in U}V_{F}(x+v+u)=[0,\infty)$$
    where the last equality follows from Lemma~\ref{lem: non degeneracy degree}.
\end{proof}

In our discussion we will fix a lattice $x$ and a probability measure 
$\mu$ on the torus $\bR^d/x$ and try to say something about 
$\mu\pa{\seta{y\in \pi^{-1}(x): y \textrm{ is }DV_F}}$. 
\begin{defn}
    \label{def: a.s.dv}
    Given $x\in X$ and a probability measure $\mu$ on $\bR^d/x$, we say that $x$ is $\mu$-almost surely $DV_F$ if $\mu$-almost any $y$ is $DV_F$. 
\end{defn}

The following is one of the main result of the paper. The readr should recall Definition~\ref{defn: kdiv lattice}.
\begin{thm}\label{thm: main theorem F formulation}
    Assume $\gcd(m,n)=1$. 
    If $x\in X$ has an accumulation sequence of length $d,$
    in other words $x$ is not $k$-divergent for any $0\le k\le d-2$, then 
    $x$ is $\mu$-almost surely grid $DV_F$
    with respect to any algebraic measure $\mu$ on $\pi^{-1}(x)$. In fact, $\mu$-almost any $y\in \pi^{-1}(x)$ satisfies that the orbit closure $\overline{\seta{h_ty:t\ge 0}}\subset Y$ contains a coset of a $d-1$-dimensional subtorus of $\bR^d/x_{d-1}$. 
\end{thm}
We now deduce Theorem~\ref{thm: 1 dim in theta} from Theorem~\ref{thm: main theorem F formulation} and explain how to link the discussion about $DV_F$ grids to the discussion about the set of badly approximable targets $\on{Bad}_A$.
\begin{proof}[Proof of Theorem \ref{thm: 1 dim in theta} assuming Theorem
\ref{thm: main theorem F formulation}]
    Let $A,p_0,q_0, \eta$ be as in the statement of Theorem~\ref{thm: 1 dim in theta}. 
    Consider the lattice $x_A$ defined in~\eqref{eq: xA} and the one dimensional algebraic measure
    $\mu$ on $\bR^d/x_A$ supported on the 1-dimensional subtorus 
    given by the $x_A$-rational subspace
    $$U:= \seta{\mat{t(Aq_0+p_0)\\tq_0} :t\in \bR}$$
    and translation $\mat{-\eta\\ 0}.$ 
    Theorem~\ref{thm: main theorem F formulation} says that $\mu$-almost any grid is $DV_F$. In other words, for Lebesgue almost any $t\in \bR,$
    $F(x_A- \mat{t(Aq_0+p_0)+\eta\\tq_0})$ is dense in $[0,\infty).$
    In particular, for Lebesgue almost any $t$, 
    $F$ attains arbitrarily small positive values on the grid 
    $$x_A- \mat{t(Aq_0+p_0)+\eta\\tq_0}=\seta{\mat{p+Aq-t(Aq_0+p_0)-\eta\\ q-tq_0}: q\in\bZ^n, p\in \bZ^m}.$$
    For a vector $\mat{p+Aq-t(Aq_0+p_0)-\eta\\ q-tq_0}$ in the above set we have 
    $$F(\mat{p+Aq-t(Aq_0+p_0)-\eta\\ q-tq_0}) = \norm{q-tq_0}^n
    \norm{p+Aq-t(Aq_0+p_0)-\eta}^m.$$
    This quantity cannot become positive and arbitrarily small using finitely many $q$'s and so we get that there exists a sequence $q_i\in \bZ^m$ and 
    $p_i\in \bZ^n$ such that $\norm{q_i}\to\infty$ and 
    $\norm{q_i}^n\norm{p_i+Aq_i-t(Aq_0+p_0)-\eta}^m $ is a sequence of positive numbers going to zero. In particular
    $$\liminf_{q\in \bZ^n,\norm{q}\to\infty}\norm{q}^{n/m}\idist{Aq-(-t(Aq_0+p_0)+\eta)} = 0$$ 
    as desired.
   
\end{proof}

\subsection{An intriguing example}
We end this section with the following example which shows that the assumption
$\gcd(n,m)=1$ is not entirely an artifact of the method of proof (see Lemma~\ref{lemma: zero is not a weight} and Corollary~\ref{cor: subspaces do not stabilize}, for the point where this assumption enters our discussion). 
It shows that the statement of Theorem~\ref{thm: main theorem F formulation} can fail drastically when $n=m\ge 2$.
\begin{exmp}
    \label{exmp: counter n=m}
    Take any 2-lattice in the plane $\smallmat{a&b\\c&d}\bZ^2$ and find a translate $\smallmat{s_1\\ s_2}$
    satisfying 
    \begin{equation}
        \label{eq: lower bound in 2d}
    \inf_{q_1,q_2\in \bZ}\av{aq_1+bq_2+s_1}\cdot\av{cq_1+dq_2+s_2}:=c>0. 
    \end{equation}
    The existence of such $s_1,s_2$ is tightly related to Theorems 
    \ref{thm: Bad non empty}, \ref{thm: Bad is winning}. 
    A proof for this existence can be found in
    \cite[Theorem 1]{D}. 
    
    Now choose $n=m$ for any positive integer $m$ and take a matrix $g\in \SL_d(\bR)$ whose $1$st and $m+1$'th rows are
    \begin{align*}
        &(a,0,\dots,0,b,0,\dots,0)\\
        &(c,0,\dots,0,d,0,\dots,0)
    \end{align*}
    where the $b$ and the $d$ are the $m+1$-coordinates. 

    Let us denote by $\mb{p}:\bR^d\to \bR^2$, the projection on the $1$'st and $m+1$'th coordinates. For such a choice of $g$, 
    for any grid of the form $y= g\bZ^d+ u$, where
    \begin{equation}
        \label{eq: the form of u}
        u=(s_1,*,\dots,*, s_2,*,\dots,*)^{\on{tr}}
    \end{equation}
    we have that 
    $$\mb{p}(y) = \mat{a&b\\c&d}\bZ^2+\mat{s_1\\s_2}.$$
    Since we have the lower bound, for $\smallmat{v\\w}\in \bR^d$,
    $$F(\smallmat{v\\w})=\norm{v}^m\norm{w}^m\ge \av{v_1}^m\av{w_1}^m$$
    we deduce from \eqref{eq: lower bound in 2d} that 
    $$\inf V_F(y)\ge c^m>0$$
    and in particular $y$ is not $DV_F$. Note that this is where it is important that $n=m$. By the general shape \eqref{eq: the form of u} of the translate $u$ giving the grid $y$, this shows that the set of grids of 
    $g\bZ^d$ which are not $DV_F$ contains a coset of a codimension $2$ subtorus.

    We note that it is not hard to make the choices in the above general construction in a way that the $h_t$-orbit of $g\bZ^d$ is \textbf{bounded} and in particular, this lattice is not $k$-divergent for any $k$. This shows that Theorem~\ref{thm: main theorem F formulation} fails drastically when $n=m$. It also serves as a counter example to 
    \cite[Theorem 2.3]{GDV} for the regime $n=m\ge 2$. 
    The counter example we present in Theorem~\ref{thm: counterexample intro} is more sophisticated and deals with the regime $n=1$ and $m>2$. 

    A possible choice of the parameters is given by choosing $n=m=2$, $d=4$,
    $$g = \al \mat{1&0&\sqrt{2}&0\\ 0&1&0&\sqrt{2}\\
    1&0&-\sqrt{2}&0\\ 0&1&0&-\sqrt{2}},$$
    where $\al$ is chosen so that $\det g =1$. We shall ignore the multiplicative factor $\al$ below.  
    To see that the orbit $h_t g\bZ^4$ is bounded, we use Mahler's criterion which says that a set of lattices in $X$ has compact closure if and only if there is a uniform lower bound on the lengths of non-zero vectors in lattices in the set. Here, the general form of a non-zero vector in a lattice in $\seta{h_t g\bZ^4:t\in \bR}$ is of the form 
    $$\mat{ 
    e^t (q_1 + q_3\sqrt{2})\\
    e^t (q_2 +q_4\sqrt{2})\\
    e^{-t}(q_1-q_3\sqrt{2})\\
    e^{-t}(q_2-q_4\sqrt{2})}$$
    where $(q_1,q_2,q_3,q_4)\ne (0,0,0,0)$. Assuming for concreteness that $(q_1,q_3)\ne (0,0)$ we get that the length of the above $4$-dimensional vector is bounded below by the length of the 2-dimensional vector
    $$\mat{e^t (q_1 + q_3\sqrt{2})\\ e^{-t}(q_1-q_3\sqrt{2})}.$$
    Now this is a vector whose product of coordinates equals $q_1^2-2q_3^2$, which is a non-zero integer and so it is impossible that both of the coordinates are smaller than $1$ in absolute value. This shows that the length of the 4-dimensional vector above is at least 1 as well.

\end{exmp}
\subsection{The mistake in \cite{GDV}}\label{sec: the mistake}
The main result in \cite{GDV} is Theorem $2.2$ there. This theorem is solid and has interesting applications beyond the discussion of the present paper, namely, for functions other than $F$. The mistake in that paper 
is in sections $2.2$ and $3.2$ where the main theorem was applied to the function $F$ (denoted there by $P_{n,m}$) under the false statement that 
$F$ has non-degeneracy degree 1 according to Definition~\ref{defn: non degeneracy degree} of the present paper. We now comment on the validity/falsity of each of the statements 
made in \cite[\S 2.2, \S 3.2]{GDV}. 
\begin{itemize}
    \item \cite[Theorem 2.3]{GDV} is not true in the generality stated. It is true for the Haar measure on $\bR^d/x$ as we prove in Corollary
    \ref{cor: Haar mixing} in this paper. On the other hand, we saw in Example \ref{exmp: counter n=m} and in Theorem~\ref{thm: counterexample intro} that  \cite[Theorem 2.3]{GDV} is not true for some choices of $n,m$ and algebraic measures of 
    co-dimension 2. We do not know at this point if the statement holds for co-dimension 1 algebraic measures. In particular, the simplest case that remains open is the following:
    \begin{ques}\label{ques: is it true}
        In the notation of the current paper, let $d=3, m=2, n=1$. Let $x\in X$ be a lattice with a non-divergent $h_t$-orbit. Is it true that for any non-trivial algebraic measure $\mu$ on $\bR^3/x$, we have that $x$ is $\mu$-almost surely $DV_F$? Is it true that for $\mu$-almost any grid $y$, $0\in \overline{V_F(y)}$? Specializing to vectors, given a non-singular column vector $\theta\in \bR^2$, is it true that 
        $\on{Bad}_\theta$ is a null set with respect to any algebraic measure 
        on $\bT^2$?
    \end{ques}
    \item \cite[Theorem 3.4]{GDV} is true as stated. This follows from Corollary~\ref{cor: Haar mixing} and Theorem~\ref{thm: Uri lebesgue} of 
    the present paper.
    \item \cite[Theorem 3.7]{GDV} is false in the generality stated. Theorem~\ref{thm: counterexample intro} shows that it is false for some algebraic measures.
    \item
    \cite[Corollary 3.8]{GDV} is under question. We do not yet have a counter example. The statement does hold if one replaces the non-singularity assumption by the stronger assumption that the vector is not 
    $k$-divergent for $0\le k\le d-1$. This is a special case of Theorem~\ref{thm: 1 dim in theta}.  
    \end{itemize}
\section{Generalities about measures}
\subsection{Borel measures and weak* convergence}\label{sec: pushforwards}
In this section we present the elementary theory we will use regarding regular Borel measures on topological spaces. Although the discussion is widely known, we include it here for the sake of completeness which some readers might appreciate.

Throughout this section we assume all topological spaces to be locally compact second countable Hausdorff spaces. This puts us in a setting in which every Borel measure which gives finite measure to compact sets is automatically regular. Such measures are called Radon measures. Furthermore, we will only discuss finite measures so the condition one finiteness on compacta will be automatic.

In later sections, when we apply the theory presented here, we will only be concerned with the spaces $Y,X$ and their subspaces. Nevertheless, we try to keep the discussion here abstract for clarity.

\begin{defn}[Topology on regular complex Borel measures]\label{defn: Borel measures}
    Let $Z$ be locally compact second countable Hausdorff space. 
    Consider Banach space $C_0(Z,\Co)$ of complex continuous functions vanishing at infinity with the supremum norm. Its dual, by \cite[Theorem 6.19]{Ru-RACA}, is isometric to $\mathcal{M}(Z)$, the vector space of complex Radon measures on $Z$ endowed with the norm induced by total variation (cf. \cite[Chapter 6]{Ru-RACA}). The bilinear duality pairing is given by integration:
    \begin{equation*}
        (\phi,\mu) \mapsto \int_Z \phi d\mu.
    \end{equation*}
We endow $\mathcal{M}(Z)$ with the weak* topology. That is, the smallest topology making the all the functionals
\begin{equation*}
    \mu \mapsto \int_Z \phi d\mu,\ (\text{where } \phi \in C_0(Z,\Co))
\end{equation*}
continuous.
We denote by $\mathcal{M}_1(Z)$ the set of measures with norm $\leq 1$, the set of real positive measures by $\ca{M}_+(Z)$, and the set of real positive measures of norm $1$ by $\ca{P}(Z)$. This last set is the set of Borel probabilities on $Z$.
We have a partial order on $\ca{M}(Z)$ given by 
\begin{equation*}
    \left(\mu \geq  \nu\right) \iff \left(\mu - \nu \in \ca{M}_+(Z)\right).
\end{equation*}
\end{defn}
Given $\mu \in \ca{P}(Z)$, we define the \textit{support} of $\mu$, $\on{supp} \mu$ to be the complement of the set
\begin{equation*}
    \bigcup\left\lbrace V \subset Z: V \text{ is open and } \mu(V)=0 \right\rbrace.
\end{equation*}
\begin{lem}\label{lem: measures are equal if they agree on dense set}
    Let $Z$ be second countable locally compact Hausdorff space and let $\mu, \nu \in \mathcal{M}(Z)$. If $D \subset C_0(Z, \Co)$ spans a dense subspace of functions, then 
    \begin{equation*}
        \left(\mu = \nu\right) \iff \left(\int_Z \phi d\mu = \int_Z \phi d\nu \text{ for all } \phi \in D \right).
    \end{equation*}
\end{lem}
\begin{proof}
    The implication $\Longrightarrow$ is trivial. The other implication follows from the fact that $\mu,\nu$ are continuous linear functionals on $C_0(Z,\bC)$.
    \ignore{
    Since integration is bilinear (over functions and measures), it suffices to prove the Lemma when $\nu=0$. The forward direction is clear. We prove the other direction. Let $\psi \in C_0(Z,\Co)$ be arbitrary. Let $\varepsilon>0$ be arbitrary.
    We compute, for $\phi\in D$,
    \begin{equation*}
        \begin{split}
            \left| \int_{Z} \psi d\mu \right| &= 
            \left| \int_Z (\psi - \phi) d\mu \right|
            \\
            &\leq \|\psi - \phi\| \cdot \|\mu\|
        \end{split}
    \end{equation*}
    where $\|\mu\|$ denotes norm in $\ca{M}(Z)$ and $\|\psi - \phi\|$ denotes the norm in $C_0(Z,\Co)$. 
    Thus, if we choose $\phi \in D$ which has
    \begin{equation*}
        \|\psi - \phi\| < \varepsilon \|\mu\|^{-1}
    \end{equation*}
    (assuming that $\mu$ is nonzero), we get
    \begin{equation*}
        \left|\int_Z \psi d\mu \right| < \varepsilon.
    \end{equation*}
    By the duality of $\mathcal{M}(Z)$ and $C_0(Z,\Co)$, we are done.
    }
\end{proof}
An example to keep in mind for the applicability of the above lemma is when 
$Z=\bT^m$ and $D$ is the set of characters $D=\seta{e^{2\pi i\idist{x,p}}:p\in \bZ^m}.$
\ignore{With respect to the preceeding Definition and Lemma, we remark there can exist nonzero, nonregular Borel measures on $Z$ which give rise to the zero functional on $C_0(Z,\Co)$. See \cite[Chapter 2, exercise 18]{Ru-RACA}.
}
\begin{lem}\label{lem: B-A}
    Let $Z$ be a locally compact second countable Hausdroff space. In the weak* topology, the unit ball $\ca{M}_1(Z)$ is compact. Furthermore, when $Z$ is compact, $\cP(Z)$ is compact.
\end{lem}
\begin{proof}
    The first statement is the Banach-Alaoglu theorem. See \cite[Theorem 3.15]{Ru-FA}. For the second statement, note that 
    $$\cM_+(Z) = \cap_{\phi\ge 0, \phi\in C_0(Z,\bC)}\seta{\mu:\int_Z\phi d\mu\ge 0}$$ 
    is weak* closed and when $Z$ is compact, $\mb{1}_Z$ (the constant function $1$,) is an element of $C_0(Z,\bC)$ and so the subset $\cP(Z)= \seta{\mu\in \cM_+(Z):\int_Z\mb{1}_Z d\mu = 1}$ is also closed. Therefore, $\cP(Z)\subset \cM_1(Z)$ is a closed subset and hence compact (note that the weak* topology is Hausdorff). 
\end{proof}
\begin{lem}\label{lem: unit ball is metrizable}
    Let $Z$ be a locally compact second countable Hausdorff space. Then the unit ball $\mathcal{M}_1(Z)$ with the weak* topology is metrizable.
\end{lem}
\begin{proof}
    Under our assumptions $C_0(Z,\Co)$ contains a countable dense set. Call it $(\phi_i)_{i \in \N} \subset C_0(Z,\Co)$. We leave it to the reader to check that the topology on $\mathcal{M}_1(Z)$ induced by the metric
    \begin{equation*}
        \dist(\mu, \nu) := \sum_{i=1}^\infty \frac{1}{2^i \|\phi_i\|} \left| \int \phi_i d(\mu - \nu) \right|
    \end{equation*}
coincides with  topology on $\mathcal{M}_1(Z)$ induced by the weak* topology on $\mathcal{M}(Z)$.
Alternatively, apply Lemmas \ref{lem: B-A} and \ref{lem: measures are equal if they agree on dense set} to use \cite[(c) of page 63]{Ru-FA}.
\end{proof}

\begin{defn}[Pushforwards of measures]\label{defn: pushforward}
    Let $Z$ and $Z'$ be locally compact second countable Hausdorff spaces. If $f: Z \to Z'$ is a Borel measurable map between them and $\mu \in \mathcal{M}(Z)$, we have the pushforward of $\mu$ by $f$ in $\mathcal{M}(Z')$ defined by the formula:
    \begin{equation*}
        (f_*\mu) (E) := \mu(f^{-1}(E))\ \text{ for every Borel } E \subset Z'.
    \end{equation*}
    It is straightforward to check that for any bounded measurable function $\phi: Z' \to \Co$ we have the equality
    \begin{equation*}
        \int_{Z'} \phi d(f_* \mu) = \int_Z \phi \circ f d\mu,
    \end{equation*}
    and that this formula characterizes
    $f_*\mu$.
\end{defn}
    \ignore{Note, in the above definition, complex Borel measures are automatically regular and Radon by \cite[Theorem 2.18]{Ru-RACA}.}

\begin{lem}\label{lem: functoriality of pushforwards}
    Let $Z, Z', Z''$ be locally compact second countable Hausdorff spaces. Let $\mu \in \mathcal{M}(Z)$, and say we have measurable maps $f:Z \to Z'$ and $g: Z' \to Z''$.
    Then,
    \begin{equation*}
        (g\circ f)_*(\mu) = g_* (f_* \mu) \ \text{ in } \mathcal{M}(Z'').
    \end{equation*}
\end{lem}
\begin{proof}
    This follows at once from Definition \ref{defn: pushforward}.
\end{proof}
\begin{lem}\label{lem: pushforward is continuous}
    Let $Z, Z'$ be locally compact second countable Hausdorff spaces and let $f : Z \to Z'$ be continuous and proper.
    Then, 
    \begin{equation*}
        f_* : \ca{M}(Z) \to \ca{M}(Z')
    \end{equation*}
    is continuous with respect to  the weak* topologies.
\end{lem}
\begin{proof}
    It is clear that $f_*$ is a linear map. 
    It is thus enough to show that preimages of basic open open neighbourhoods of $0$ are open. A basic weak* open neighbourhood 
    of $0\in \cM(Z')$ is given by, for $\eps>0, \phi \in C_0(Z',\bC)$,  $\seta{\nu:\av{\int \phi d\nu}<\eps}.$ 
    It follows from the definition of $f_*$ that for $\phi \in C_0(Z,\Co), \eps>0$,
    \begin{equation*}
        (f_*)^{-1} \left\lbrace \nu \in \ca{M}(Z') : \left| \int \phi d\nu \right| < \eps \right\rbrace = \left\lbrace \mu \in \ca{M}(Z): \left| \int \phi \circ f d\mu \right|<\eps\right\rbrace.
    \end{equation*}
    Since $f$ is continuous and proper, $\phi\circ f\in C_0(Z,\bC)$ and so the right hand side is a basic open neighbourhood of $0$ as well. 
\end{proof}
The following Lemma is used to justify a step in a particular argument. The reader should skip it and return to it when referred to. 
\begin{lem}\label{lem: convergence transitive}
    Let $Z, Z', Z''$ be locally compact second countable Hausdorff spaces. Let $(f_l)_{l \in \N}$ and $(g_l)_{l \in \N}$ be sequences of continuous proper functions from $Z \to Z'$ and $Z' \to Z''$ respectively.
    Let $\mu \in \mathcal{M}(Z)$, $\nu \in \ca{M}(Z')$ and $\eta \in \ca{M}(Z'')$ be measures and assume we have the weak* convergence
    \begin{equation*}
        \lim_{n \to \infty} (f_l)_* \mu = \nu\ \text{ and }\ \lim_{l \to \infty} (g_l)_* \nu = \eta.
    \end{equation*}
    Then, there are subsequences $(p_l)_{l \in \N}$ and $(q_l)_{l \in \N}$ of $\N$ such that we have the weak* convergence
    \begin{equation*}
        \lim_{l \to \infty} (g_{p_l} \circ f_{q_l})_* \mu = \eta.
    \end{equation*}
\end{lem}
\begin{proof}
    Since the maps on measures $(f_l)_*$ and $(g_l)_*$ have operator norm less than or equal to $1$, we might as well assume $\mu, \nu$ and $\eta$ are in the unit balls of their respective measure spaces.
    
    By Lemma \ref{lem: unit ball is metrizable}, the weak* topologies on the unit balls $\ca{M}_1(Z)$,  $\ca{M}_1(Z')$ and $\ca{M}_1(Z'')$ are metrizable. 
    Let $l\in \N$ and consider the metric ball
    \begin{equation*}
        B_{\ca{M}_1(Z'')}(\eta, n^{-1}) \subset \ca{M}_1(Z'')
    \end{equation*}
    with center $\eta$ and radius $l^{-1}$ in this metric.
By convergence of $((g_l)_* \nu)_{l \in \N}$, there exists some $p_l \in \N$ with
\begin{equation*}
    (g_{p_l})_* \nu \in B_{\ca{M}_1(Z'')}(\eta, l^{-1}).
\end{equation*}
By Lemma \ref{lem: pushforward is continuous}, there is some $\delta>0$ with
\begin{equation}\label{eq: nu in}
    \nu \in B_{\ca{M}_1(Z')}(\nu, \delta) \subset (g_{p_l})_*^{-1} \left( B_{\ca{M}_1(Z'')}(\eta, l^{-1}) \right).
\end{equation}
By the convergence of $((f_l)_* \mu)_{l \in \N}$, we see that there exists $ q_l \in \N$ with
\begin{equation}\label{eq: mu in}
    (f_{q_l})_* \mu \in B_{\ca{M}_1(Z')}(\nu, \delta).
\end{equation}
Combining equations \eqref{eq: nu in} and \eqref{eq: mu in}, we see that 
\begin{equation*}
    (g_{p_l})_* \circ (f_{q_l})_* \mu \in B_{\ca{M}_1(Z'')}(\eta, l^{-1}).
\end{equation*}
Varying $l \in \N$ gives us our sequence.
\end{proof}
\ignore{
\begin{lem}\label{lem: continuous maps are continuous}
    Let $Z, Z', Z''$ be locally compact second countable Hausdorff spaces. Let $(f_n)_{l \in \N}$ be a sequence of measurable maps maps from $Z \to Z'$, let $\mu \in \mathcal{M}(Z)$ and let $g : Z' \to Z''$ be a continuous proper map.
    If we have the weak* convergence
    \begin{equation*}
        \lim_{n \to \infty} (f_n)_* \mu = \nu \ \text{ in } \mathcal{M}(Z'),
    \end{equation*}
    then we also have the weak* convergence
    \begin{equation*}
        \lim_{n \to \infty} (g\circ f_n)_* \mu = g_* \nu \ \text{ in } \mathcal{M}(Z'').
    \end{equation*}
\end{lem}
\begin{proof}
This follows from Lemmas \ref{lem: functoriality of pushforwards}, \ref{lem: pushforward is continuous}.
\end{proof}
}
The following Lemma is used to justify a step in a particular argument. The reader should skip it and return to it when referred to. 
\begin{lem}\label{lem: weak limit is the same under perturbations}
    Say $Z$ and $Z'$ are locally compact second countable metric spaces with distance functions $\dist_{Z}$ and $\dist_{Z'}$ respectively. Let $\mu \in \mathcal{M}(Z)$.
    Suppose we have sequences of continuous maps $(f_l)_{l \in \N}$,  $(g_l)_{l \in \N}$ from $Z \to Z'$ with the property:
    \begin{equation*}
        \lim_{l \to \infty}\ \sup_{z \in Z} \dist_{Z'}(f_l(z), g_l(z)) = 0.
    \end{equation*}
    If there exists a weak* limit
    \begin{equation*}
        \lim_{l \to \infty} (f_l)_* \mu = \nu\ \text{ in } \mathcal{M}(Z')
    \end{equation*}
    then we also have the same weak* limit
    \begin{equation*}
        \lim_{l \to \infty} (g_l)_* \mu = \nu.
    \end{equation*}
\end{lem}
\begin{proof}
We can assume $\mu \neq 0$. Let $\phi \in C_0(Z', \Co)$ 
\ignore{be  compactly-supported} 
and let $\varepsilon>0$. 
Since $\phi$ is uniformly continuous, there is a $\delta>0$ such that
\begin{equation*}
    \left(\dist_{Z'}(z'_1,z'_2) < \delta\right) \implies \left( \left|\phi(z'_1) - \phi(z_2')\right| < \frac{\varepsilon}{2} \left( \|\mu\|\right)^{-1} \right).
\end{equation*}
Choose $l_1\in \N$ large enough so that, for all $l > l_1$,
\begin{equation*}
    \sup_{z \in Z'} d_{Z'}(f_l(z), g_l(z)) < \delta.
\end{equation*}
Choose $l_2 \in \N$ large enough so that, for all $l > l_2$,
\begin{equation*}
    \left|\int_{Z'} \phi d (f_l)_* \mu  - \int_{Z'} \phi d\nu \right| < \frac{\varepsilon}{2}.
\end{equation*}
For $n> \max\{l_1, l_2\}$, we compute
\begin{equation*}
    \begin{split}
        \left|\int_{Z'} \phi d(g_l)_* \mu - \int_{Z'} \phi d\nu  \right| 
        &= \left| \int_{Z} \phi\circ g_l d\mu - \int_{Z'} \phi d\nu\right|
        \\
        &= \left|\int_{Z} \left(\phi \circ g_l - \phi \circ f_l)\right)d\mu + \int_{Z'} \phi d(f_l)_*\mu - \int_{Z'} \phi d\nu \right| \\
        &\leq \frac{\varepsilon}{2} + \frac{\varepsilon}{2}.
    \end{split}
\end{equation*}
\ignore{
Now, if $\psi \in C_0(Z',\Co)$ is arbitrary and $\varepsilon>0$, choose $\phi \in C_0(Z',\Co)$ which is compactly supported and satisfies 
\begin{equation}\label{eq: total variation estimate}
    \sup_{z' \in Z'} |\phi(z') - \psi(z')| < \frac{\varepsilon}{3} \max \left\lbrace \|\mu\|^{-1},  \|\nu\|^{-1}  \right\rbrace.
\end{equation}
If $\nu = 0$, we take the right hand side of \eqref{eq: total variation estimate} as $(\varepsilon \|\mu\|^{-1})/3$.
For $l\in \N$, we compute
\begin{multline*}
    \left| \int_{Z'} \psi d(g_l)_*\mu - \int_{Z'} \psi d\nu \right|
        \\
        =\left|\int_{Z} \left(\psi \circ g_l - \phi \circ g_l \right) d\mu\right| + \left|\int_{Z'} \phi d(g_l)_*\mu - \int_{Z'} \phi d\nu \right|  + \left|\int_{Z'}(\phi - \psi)d\nu \right|.
\end{multline*} 
The first part of the argument ensures that the middle term is less than $\varepsilon/3$ for all sufficiently large $n$. The terms on either are can be estimated using \eqref{eq: total variation estimate} and so we are done.
}
\end{proof}
\subsection{Mixing convergence}\label{sec: mixing abstract}
A key concept which plays a role in the proof of Theorem~\ref{thm: main theorem F formulation} is 
the notion of mixing convergence of measures which was introduced in \cite[Definition 4.1]{GDV}. We review this concept for completeness.
\begin{defn}[Mixing convergence]\label{defn: mixing}
    Let $Z$ and $Z'$ be locally compact second countable Hausdorff spaces. Let $\mu \in \ca{P}(Z)$ and let $\nu \in \ca{P}(Z')$. Let $(f_l)_{l \in \N}$ be a sequence of measurable maps from $Z \to Z'$.
     We say $(f_l)_{l \in \N}$ mixes $\mu$ to $\nu$ if, for every absolutely continuous probability measure $ \eta \ll \mu$ in $\ca{P}(Z)$ we have that
    \begin{equation}\label{eq: mixing conv}
        \lim_{l \to \infty} (f_l)_* \eta =\nu.
    \end{equation}
\end{defn}
Note that by taking $\eta=\mu$ we see that mixing convergence is stronger that
weak* convergence. 
We will use the following reformulation of mixing convergence.
\begin{prop}\label{prop: mixing functions}
In the notation of Definition \ref{defn: mixing}, $(f_l)_*$ mixes $\mu$ to $\nu$ if and only if
\begin{equation}\label{eq: mixing}
    \lim_{l \to \infty} \int_{Z} \phi(f_l(z)) \psi(z) d\mu(z) = \left(\int_{Z} \psi d\mu \right) \left(\int_{Z'}\phi d\nu \right) 
\end{equation}
for all $\phi \in C_0(Z',\Co), \psi \in L^1(Z,\mu)$.

Further, the condition of \eqref{eq: mixing} holds if it does on subsets of $L^1(Z,\mu)$ and $C(Z',\Co)$ spanning dense subspaces.
\end{prop}
\begin{proof}
Consider the bilinear map $\idist{\cdot, \cdot}_l:C_0(Z',\bC)\times L^1(Z,\mu)\to \bC$ given by
$$\idist{\phi,\psi}_l=\int_Z\phi(f_l(z))\psi(z)d\mu(z).$$
By the Radon-Nikodym theorem, the condition $\eta\ll \mu$, $\eta\in \cP(Z)$, in Definition~\ref{defn: mixing} is equivalent to the fact that there exists a density $\psi\in L^1(Z,\mu)$ such that $\eta = \psi d\mu$ and $\psi\ge0$, $\int_Z\psi d\mu=1.$
Interpreting the definition of weak* convergence we see that Definition~\ref{defn: mixing} is equivalent to requiring that equation
\eqref{eq: mixing} holds for any $\phi\in C_0(Z',\bC)$ and any 
$\psi\in L^1(Z,\mu)$ with $\psi\ge0$ and $\int_Z \psi d\mu =1$.
Since both sides of \eqref{eq: mixing} are linear in $\psi$ and $L^1(Z,\mu)$ is spanned by positive functions of integral 1, this equation must hold for all
$\psi\in L^1(Z,\mu)$ as claimed in the proposition. 

For the last sentence in the proposition we note the inequality
\begin{equation}\label{eq: cont of bili}
\av{\idist{\phi,\psi}_l}\le \norm{\phi}_\infty \cdot \norm{\psi}_{L^1(Z,\mu)}.
\end{equation}
If we know~\eqref{eq: mixing} for collections $D_1\subset C_0(Z',\bC), D_2\subset L^1(Z,\mu)$, then the bilinearity of $\idist{\cdot,\cdot}_l$ and the inequality \eqref{eq: cont of bili} allows to propagate \eqref{eq: mixing} to the closures of the spans of $D_1,D_2$ in the corresponding Banach space.

\end{proof}
The following Proposition is the source of the \textit{almost any} part in
our main Theorem~\ref{thm: main theorem F formulation} and its derivatives. 
\begin{prop}\label{prop: mixing implies density}
    Let $Z$ and $Z'$ be locally compact second countable Hausdorff spaces. Let $\mu \in \ca{M}(Z)$ and let $\nu \in \ca{M}(Z')$. Let $(f_l)_{l \in \N}$ be a sequence of continuous maps from $Z \to Z'$ which mixes $\mu$ to $\nu$.
    Then, for $\mu$ almost every $z \in Z$, we have
    \begin{equation*}
        \overline{\left\lbrace f_l(z): l \in \N \right\rbrace} \supset \on{supp} \nu.
    \end{equation*}
\end{prop}
\begin{proof}
As $\on{supp} \nu\subset Z'$, it has a countable dense set $D$.
It is enough to show that for any $z'\in D$, $\mu$-almost any $z$ satisfies that $\seta{f_l(z)}$ visits any open neighbourhood of $z$. Since there is a countable base for the topology of $Z'$, it is enough to fix a neighbourhood 
$U$ of $z'$ and show that $\mu$-almost any $z$ satisfies that for some $l$, $f_l(z)\in U$. By way of contradiction, there exists $z'\in D$ and 
an open neighbourhood $U$ of $z'$ such that the set 
$$S = \seta{z\in Z:\forall l, f_l(z)\notin U}$$
has positive $\mu$-measure.
\ignore{

Thus, to prove the proposition it suffices to show that, for every $z'  \in \on{supp}\nu$, $i \in \N$ and $\varepsilon>0$, the $\mu$-measure of
\begin{equation*}
     S= \bigcap_{n > i} \left\lbrace z \in Z: \dist_{Z'}(f_n(z), z') > \varepsilon\right\rbrace
\end{equation*}
is zero.
}
Consider the nonzero indicator function $\mathds{1}_S \in L^1(\mu)$ the probability measure $\eta:= \mu(S)^{-1}\mathds{1}_S d\mu$ which is absolutely continuous with respect to $\mu$. By definition of mixing convergence we have that $(f_l)_*\eta\to \nu$. In particular, if we choose a non-zero function 
$\phi\in C_c(Z',\bC)$ with compact support contained in $U$ and such that $\phi\ge 0$ and $\phi(z')>0$, then the above weak* convergence implies
\begin{equation*}
    \lim_{l \to \infty} \mu(S)^{-1}\int_Z \phi(f_l(z)) \mathds{1}_S(z) d\mu (z) = \int_{Z'} \phi d\nu. 
\end{equation*}
However, the left hand side is zero all $l$ by the definition of $S, U,$ and $\phi$, while the integral on the right hand side is positive since $z' \in \on{supp} \nu$ and $\phi$ was chosen non-negative and positive at $z'$. We arrive at a contradiction and conclude the proof.
\end{proof}
\section{Proof of Theorem~\ref{thm: Uri lebesgue}}\label{sec: proof for Lebesuge}
In this section we reproduce (as promised in the introduction) the proof of 
Theorem~\ref{thm: Uri lebesgue} which appeared implicitly in \cite{GDV}. The proof relies on the notion of mixing convergence of measures and elementary theory of characters.
Although Theorem~\ref{thm: Uri lebesgue} is not essential for the proof our main result Theorem~\ref{thm: main theorem F formulation}, its proof is a simplified version of the one for Theorem~\ref{thm: main theorem F formulation} and prepares the reader for the discussion in 
\S\ref{sec: long sequences}.
\begin{defn}[Tori and characters]\label{defn: characters on tori}
    Let $d \in \N$. If $x \subset \R^d$ is a lattice, a character on $\R^d/x$ is a continuous group homomorphism
    \begin{equation*}
        \chi : \R^d/x \to \Co\smallsetminus \{0\}.
    \end{equation*}
    Linear combinations of characters are dense in $C(\R^d/x,\Co)$. See for example \cite[4.24]{Ru-RACA}.
If $\mu \in \mathcal{M}(\R^d/x)$, Lemma \ref{lem: measures are equal if they agree on dense set} and the linearity of integrals shows that $\mu$ is uniquely determined in $\mathcal{M}(\R^d/x)$ by the values 
\begin{equation}\label{eq: hats}
    \left\lbrace \int_{\R^d/x} \chi d\mu \in \Co: \chi \text{ is a character}\right\rbrace.
\end{equation} 
When $\mu$ is the Haar probability on $\R^d/x$, the set of values in \eqref{eq: hats} is $1$ if $\chi$ is the identity character and $0$ otherwise.
On $\R^d$, we use the standard inner product $(\cdot| \cdot)$ and, for the lattice $x \subset \R^d$, we define the dual lattice
\begin{equation*}
    x^* := \left\lbrace v \in \R^d: \text{ for all } a \in x \text{ we have } (v|a) \in \Z \right\rbrace.
\end{equation*}
The set of characters of $\R^d/x$ can then be identified with $x^*$ where, for $v \in x^*$ and $w\in \R^d/x$, we have
\begin{equation*}
    \chi_v(w+x) := \exp\left(2i\pi (v|w) \right).
\end{equation*} 
    The lattice $\Z^d$ is self dual.
\end{defn}
Recall the notation $h_t$ in \eqref{eq: ht definition} and Definition \ref{defn: asymptotic limit points}.
The following proposition shows how non-divergence relates to mixing convergence. Its corollary together with the the Inheritance Lemma~\ref{lem: inheritance} relate the discussion to the notion of almost sure $DV_F$ (see Definition~\ref{defn: DVF}).
\begin{prop}\label{prop: non div implies mixing}
    Let $x_0, x_1 \in X$ be two lattices with $x_1 \in \partial(x_0)$. Accordingly, let $(t_l)_{l \in \N}$ be a divergent sequence of positive reals and let $(\varepsilon_l)_{l \in \N} \subset G$ be a sequence converging to the identity such that
    \begin{equation}\label{eq: prop Haar mixing}
        \varepsilon_l h_{t_l} x_0 = x_1 \text{ for all } l\in \N.
    \end{equation}
    Then, $\varepsilon_l h_{t_l}$ mixes the Haar measure $\lam_{\bR^d/x_0}$ to the Haar measure $\lam_{\bR^d/x_1}$.
\end{prop}
\begin{proof}
    Note, the condition \eqref{eq: prop Haar mixing} ensures that the sequence of matrices $(\varepsilon_l h_{t_l})_{l \in \N}$ induce group morphisms
    \begin{equation*}
        \R^d/x_0 \to \R^d/x_1.
    \end{equation*}  
    We use the same notation for these induced maps.
    We let $\mu$ be the Haar probability on $\R^d/x_0$, $\nu$ be the Haar probability on $\R^d/x_1$ and claim that the sequence $(\varepsilon_l h_{t_l})_{l \in \N}$ mixes $\mu$ to $\nu$.
    In order to show this we must check, in light of Proposition \ref{prop: mixing functions} and the comments in Definition \ref{defn: characters on tori}, that for every $b \in x_1^*$ and $a \in x_0^*$ and corresponding characters $\chi_a$ and $\chi_b$ on $\R^d/x_0$ and $\R^d/x_1$ respectively, we have
    \begin{equation}\label{eq: full measure mixing-1}
        \lim_{n \to \infty } \int_{\bR^d/x_0}  \chi_b(\varepsilon_l h_{t_l} (z)) \chi_a(z) d\mu(z) \overset{\mathrm{?}}{=} \left(\int_{\bR^d/x_0} \chi_a d\mu \right) \left(\int_{\bR^d/x_1} \chi_b d\nu\right).
    \end{equation}
    For the right side, we have
    \begin{equation}\label{eq: left side lebesgue mixing}
        \left(\int_{\bR^d/x_0} \chi_a d\mu \right) \left(\int_{\bR^d/x_1} \chi_b d\nu\right) =
        \begin{cases}
			1 & \text{if } a =0 \text{ and } b =0 \\
            0 & \text{if } a\neq 0\text{ or } b \neq 0
		  \end{cases}.
    \end{equation}
    For the left side, we rewrite it as
    \begin{equation}\label{eq: right side lebesgue mixing}
        \lim_{n \to \infty } \int_{\bR^d/x_0}  \exp( 2i\pi ((\varepsilon_l h_{t_l})^T b + a)| z) d\mu(z)
    \end{equation}
    where we have taken the transpose of the matrix $(\varepsilon_l h_{t_l})$.
    When $b=0$, Equations \eqref{eq: left side lebesgue mixing} and \eqref{eq: right side lebesgue mixing} clearly coincide.
    We now check equality when $b\neq 0$. 

    Fix $0\ne b\in x_1^*$ and $a\in x_0^*$. The right hand side of 
    \eqref{eq: full measure mixing-1} equals $0$. The sequence on left hand side of this equation, 
    $\int_{\bR^d/x_0} \chi_{(\varepsilon_l h_{t_l})^Tb+a}(z) d\mu(z)$,
    is a sequence of integrals of characters on $\bR^d/x_0$. As such, it
    is a sequence of $1$'s and $0$'s according to whether the character is trivial or not. We will therefore finish once we establish that the vectors $(\varepsilon_l h_{t_l})^Tb+a\in x_0^*$ are non-zero for all large enough $l$.  We have that
    $$(\varepsilon_lh_{t_l})^Tb+a = 0 \iff -h_{t_l}^{-1}a = \varepsilon_l^T b.$$
    Write $a=\smallmat{v\\w}$ with $v\in \bR^m,w\in \bR^n$. By \eqref{eq: ht definition} we see 
    that $h_{t_l}^{-1}a =\smallmat{e^{-n t_l}v\\ e^{mt_l}w}$ and since $t_l\to\infty$ this sequence of vectors either diverges or converges to zero. In any case, it is impossible that it will coincide with the sequence $\varepsilon_l^Tb\to b\ne 0$ for infinitely many $l$'s (where here we have used that $\varepsilon_l$ converges to the identity and the assumption that $b\ne 0$). 
    \ignore{
    \begin{claim}\label{claim: ht b infinity}
        When $b \neq 0$, the sequence of nonzero vectors $(\varepsilon_lh_{t_l}) ^T (b)_{l \in \N} \in \R^d$ has norm which either diverges or converges to zero.
    \end{claim}
    \begin{proof}[Proof of Claim \ref{claim: ht b infinity}]
        The sequence of vectors $(\varepsilon_l^T b)_{l \in \N}$ converges to $b$. Looking at the definition of $(h_t)_{t\in \R_{\geq 0}}$ in \eqref{eq: ht definition} and noting that the sequence $(t_l)_{l \in \N}$ is divergent, we see that the two cases of the present Claim follow according to whether $b$ has a nonzero component in its first $m$ entries or not.
    \end{proof}
    Claim \ref{claim: ht b infinity} implies that, when $b\neq 0$
    \begin{equation}\label{eq: full measure mixing}
        (\varepsilon_l h_{t_l})^T b + a \neq 0 \text{ for a cofinite set of } n \in \N.
    \end{equation}
    Indeed, if $\norm{(\varepsilon_l h_{t_l})^T b}\to\infty$, this is clear. If $(\varepsilon_l h_{t_l})^T b\to 0$, then since these vectors are non-zero, no matter what $a$ is we must have that $(\varepsilon_l h_{t_l})^T b+a\ne 0$ for all but finitely many $n$'s (in fact, since $(\varepsilon_l h_{t_l})^T b$ belongs to the dual lattice $x_0^*$, the possibility that 
    these non-zero vectors converge to $0$ is impossible). Equation
    \eqref{eq: full measure mixing} implies that for all but finitely many $n$'s the character $\exp( 2i\pi ((\varepsilon_l h_{t_l})^T b + a)| z)$
    on $\bR^d/x_0$ is non-trivial and so its integral is 0. In particular, the limit
    in \eqref{eq: right side lebesgue mixing} exists and equals 0.
    In all cases have shown that equality in \eqref{eq: full measure mixing-1} holds.
    }
\end{proof}
Applying Proposition~\ref{prop: mixing implies density} we obtain the following corollary.
\begin{cor}\label{cor: Haar mixing}
   Let $x_0\in X$ be a lattice in $\bR^d$ and assume $\seta{h_tx_0:t\ge 0}$ is non-divergent. Then, $x_0$ is $\lam_{\bR^d/x_0}$-almost surely $DV_F$. 
\end{cor} 
\begin{proof}
    By the non-divergence assumption we deduce the existence of a lattice
    $x_1\in \partial(x_0)$. As a consequence, we get the existence of a divergent sequence of reals $t_l$ and a sequence $\varepsilon_l\to e$ in $G$ satisfying $\varepsilon_lh_{t_l}x_0=x_1.$
    By Proposition~\ref{prop: non div implies mixing} we conclude that 
    $\varepsilon_lh_{t_l}$ mixes $\lam_{\bR^d/x_0}$ to $\lam_{\bR^d/x_1}$.
    An application of Proposition \ref{prop: mixing implies density} shows that, for $\lam_{\bR^d/x_0}$-almost every grid $y \in \pi^{-1}(x_0) = \R^d/x_0$, we have 
    $
        \overline{\left\lbrace \varepsilon_lh_{t_l}y : l \in \N \right\rbrace} \supset \pi^{-1}(x_1).
    $
    Since $\varepsilon_l\to e$ we deduce that
    $\overline{\left\lbrace h_{t_l}y : l \in \N \right\rbrace} \supset \pi^{-1}(x_1).$
    An application of the Inheritance Lemma~\ref{lem: inheritance} shows that 
    for $\lam_{\bR^d/x_0}$-almost any grid $y$,
    $$\overline{V_F(y)}\supset \bigcup_{v\in \bR^d} V_F(x_1+v) = F(\bR^d) = [0,\infty).$$
    This means exactly that $x_0$ is $\lam_{\bR^d/x_0}$-almost surely grid DV.
\end{proof}
We end this section with the following:
\begin{proof}[Proof of Theorem~\ref{thm: Uri lebesgue}]
    Let $A$ be as in the statement and let $x_A$ be as in~\eqref{eq: xA}. By
    Theorem~\ref{thm: Dani singular}, the non-singularity of $A$ is equivalent to the fact that $\partial(x_A)\ne \varnothing$ and so Corollary~\ref{cor: Haar mixing} applies and gives us that $x_A$ is $\lam_{\bR^d/x_A}$-almost surely $DV_F$. 
    This can be rephrased as saying that for Lebesgue almost any $\smallmat{\eta\\w}\in \bR^d$
    the grid $x_A-\smallmat{\eta\\ w}$ is $DV_F$. 
    For each such $\smallmat{\eta\\ w}$ we therefore have that 
    \begin{align*}
        F(x_A-\mat{\eta\\ w}) &= F(\seta{\mat{p+Aq-\eta \\ q-w}: p\in \bZ^m, q\in \bZ^n}\\
        &= \seta{\norm{q-w}^n \norm{p+Aq-\eta}^m: p\in \bZ^m, q\in \bZ^n}
        \end{align*}
    contains arbitrarily small positive values. Because using finitely many
    $q$'s cannot produce infinitely many arbitrarily small positive values, we deduce the existence of a sequence $q_i\in \bZ^n$ with $\norm{q_i}\to\infty$ and $p_i\in \bZ^m$ such that 
    $\norm{q_i-w}^n \norm{p_i+Aq_i-\eta}^m$ is a sequence of positive numbers approaching zero. In particular
    $$\liminf_{q\in \bZ^n \norm{q}\to\infty}\norm{q}^{n/m} \idist{Aq-\eta} = 0$$
     and so $\eta\notin \on{Bad}_A.$ This shows that $\lam_{\bT^m}(\on{Bad}_A) = 0$ as desired. 
\end{proof}
\section{Convergence of probabilities on tori and the coset lemma}\label{sec: coset}

In the previous section we proved Corollary~\ref{cor: Haar mixing} which is reminiscent of Theorem~\ref{thm: main theorem F formulation}. In it one assumes $x_0\notin \on{Div}(0)$ and deduces that $x_0$ is almost surely grid $DV_F$ with respect to the algebraic probability measure $\lam_{\bR^d/x_0}$. A key input in the proof was that push-forwards of 
$\lam_{\bR^d/x_0}$ converge mixingly (Proposition~\ref{prop: non div implies mixing}). 
Since Theorem~\ref{thm: main theorem F formulation} deals with any non-trivial algebraic probability measure on $\bR^d/x_0$, we will need to further analyze how such probabilities converge under push-forwards. 
\begin{thm}\label{thm: limits of subtori}
Let $x\in X$ be a lattice and $\bR^d/x$ be the corresponding torus. Let $\mu_l$ be a sequence of Haar probability measures on it corresponding
corresponding to the $x$-rational subspaces $V_l$. 
Assume, $\mu_l\to \nu$ in the weak* topology. Then:
\begin{enumerate} 
\item\label{i:los1} The subspace 
$V_\infty:= \bigcap_j\on{span}\pa{\bigcup_{l>j}V_l}$
is $x$-rational. 
\item\label{i:los2} $V_\infty$ is the smallest subspace containing all but finitely many of the $V_l$'s.
\item\label{i:los3} 
$\nu$ is the Haar probability measure corresponding to the $x$-rational subspace $V_\infty$.
\item\label{i:los4} $\dim \nu \ge\limsup_l\dim \mu_l$.
\item\label{i:los5} $\dim\nu  =\limsup_l\dim \mu_l$ if and only if there are infinitely many $l$'s with $V_l=V_\infty$.
\end{enumerate}
\end{thm}
\begin{proof}
\eqref{i:los1}: Each space $U_j:=\on{span}\pa{\bigcup_{l>j}V_l}$ is $x$-rational as it is spanned by $x$-rational spaces. In turn, $V_\infty$ is $x$-rational as the intersection of such spaces. 

\eqref{i:los2}: We have that $U_j$ is a descending sequence of subspaces and hence it stabilizes. Let $j_0$ be such that $U_{j_0} = V_\infty$. It follows from the definition of $U_{j_0}$ that $V_\infty$ contains all but finitely many of the $V_j$'s. Moreover, if $V$ is a subspace containing all but finitely many of the $V_j$'s, then by definition $U_j<V$ for some $j$ and so 
$V_\infty<V$. 

\eqref{i:los3}:
Because the characters $\seta{\chi_a:a\in x^*}$ span a dense subspace of 
$C_0(\bR^d/x,\bC)$, Lemma \ref{lem: measures are equal if they agree on dense set} implies that any probability measure $\mu$ on $\bR^d/x$ is characterized by its Furrier coefficients $$\seta{\hat{\mu}(a):=\int \chi_a d\mu: a\in x^*}.$$ Since characters on compact abelian groups integrate to 0 or 1 with respect to the Haar measure according to the triviality of the character, the 
subset of $\cP(\bR^d/x)$ consisting of Haar measures on subtori is characterized as follows: $\mu$ is the Haar measure corresponding to the $x$-rational 
subspace $V<\bR^d$ if and only if
\begin{equation}\label{eq: char of Haar}
\begin{array}{l}
\textrm{for any } a\in x^*,\;\hat{\mu}(a) \in\seta{0,1},\\
\textrm{and }\hat{\mu}(a)=1 \Longleftrightarrow a\in V^\perp\cap x^*.
\end{array}
\end{equation}
From the definition of $V_\infty$ and elementary properties of $\perp$, we have that 
\begin{equation}
    \label{eq: vinfty}
    V_\infty^\perp =  \bigcup_j\bigcap_{l>j}V_l^\perp.
\end{equation}
We show that $\nu$ is the Haar measure corresponding to the $x$-rational subspace $V_\infty$ using the characterization \eqref{eq: char of Haar} of Haar measures. Since characters are continuous and $\nu$ is a weak* limit of Haar measures, we have that for any $a\in x^*$, $\hat{\nu}(a)\in \seta{0,1}$. For each $a\in x^*$ we are left to check that $\hat{\nu}(a)=1$ if and only if $a\in V_\infty^\perp.$ 

We know by the weak* convergence assumption that 
$\hat{\mu}_l(a)$ converges and thus stabilizes. Thus $\hat{\nu}(a) =1$ if and only if for all large enough $l$ we have that $\hat{\mu}_l(a)=1$. Because $\mu_l$ is the Haar measure corresponding to $V_l$, this is equivalent to saying that for all large enough $l$, $a\in V_l^\perp$ which by \eqref{eq: vinfty} means that $a\in V_\infty^\perp$.

\eqref{i:los4}: As noted before, $V_\infty = \on{span}\pa{\cup_{l>j_0} V_l}$ for some $j_0$ and so the inequality $\dim \nu = \dim V_\infty \ge \limsup_l \dim V_l = \limsup \dim \mu_l$ is immediate. 

\eqref{i:los5}: From the equality 
$V_\infty=  \on{span}\pa{\cup_{l>j_0} V_l}$.
it is clear that $\dim V_\infty = \limsup_l\dim V_l$ if and only if $V_\infty = V_l$ for infinitely many $l$'s. 
\end{proof}
The following Corollary deals with limits of algebraic probability measures on a torus.
\begin{cor}
    \label{cor: limit of algebraic measures is algebraic}
    Let $x\in X$ be a lattice and let 
    $\mu_l\in \cP(\bR^d/x)$ be a sequence of algebraic measures corresponding to the $x$ rational subspaces $V_l$ and translations $w_l$. 
    If $\mu_l$ converges to $\nu$, then $\nu$ is algebraic, $\dim \nu\ge\limsup_l\dim \mu_l$, and if there is equality, for infinitely many $l$'s $V_l$ is a fixed subspace of dimension $\dim\nu$.
    \ignore{
    then $\lim_l \mu_l^0 =\nu^0$ exists and  there exists $w\in \bR^d/x$ such that $\nu  = (l_w)_*\nu^0$. Moreover $\nu$ is algebraic.
    }
\end{cor}
\begin{proof}
    It is enough to prove the statement for a subsequence.
    Denote by $\mu_l^0$ the Haar probability measures corresponding to $V_l$
    so that $\mu_l = (\ell_{w_l})_*\mu_l^0$. Here, $\ell_\cdot$ denotes translation.
    By taking a subsequence we may assume that the $\mu_l^0\to\nu^0$ for some 
    $\nu^0\in \cP(\bR^d/x)$ and furthermore that $w_l\to w$ in $\bR^d/x$.
    By Theorem~\ref{thm: limits of subtori}, $\nu^0$ is a Haar probability measure. 
    It follows from Lemma~\ref{lem: weak limit is the same under perturbations} that 
    $$\nu = \lim_l\mu_l = \lim_l(\ell_{w_l})_*\mu_l^0 = (\ell_w)_*\nu^0 $$
    which means that $\nu$ is algebraic. Moreover, Theorem~\ref{thm: limits of subtori} also gives that 
    $$\dim \nu = \dim \nu^0 \ge \limsup_l\dim V_l$$
    with equality possible only if infinitely many of the $V_l$'s are equal to the $x$-rational subspace corresponding to $\nu^0$.
\end{proof}
The next Lemma captures the following example as a general phenomenon.
\begin{exmp}
    Consider the sequence of maps $(\gamma_l)_{l \in \N} : \T^1 \to \T^2$ induced by the sequence of matrix maps
    \begin{equation*}
        \left[ {\begin{array}{c} l  \\ 
    1  \end{array}}\right] : \R \to \R^2.
    \end{equation*}
    Let $S \subset \T^2$ be the one-dimensional closed connected subgroup corresponding to $\R\ve_1$, the span of the first basis vector in $\R^2$.
    Then, for $t \in [0,1]$ which is irrational, the the set
    \begin{equation*}
        \left\lbrace \left[ {\begin{array}{c} tl  \\ 
    t  \end{array}}\right] + \Z^2 \in \T^2 : l \in \N\right\rbrace
    \end{equation*}
    contains a coset of $S$.
\end{exmp}
\begin{lem}[The Coset Lemma]\label{lem: coset}
    Let $q \in \N$ and let $(\gamma_l)_{l \in \N}$ be a sequence of continuous group morphisms from $\T^1 \to \T^q$. Let $\mu \in \ca{P}(\T^1)$ and $\nu \in \ca{P}(\T^q)$ denote the full Haar probabilities. Assume we have the weak* convergence
    \begin{equation}\label{eq: equidistribution}
        \lim_{l \to \infty} (\gamma_l)_* \mu = \nu.
    \end{equation}
    Then, there is a codimension $1$ closed connected subgroup $S\subset \T^q$ such that, for $\mu$ almost every $z \in \T^1$, 
    \begin{equation*}
        \overline{\left\lbrace \gamma_l(z) \in \T^q : l \in \N\right\rbrace} \supset S + z'\ \text{ for some } z' \in \T^q \text{ depending on $z$}.
    \end{equation*}
\end{lem}
\begin{proof}
    The Lemma is trivially true if $q=1$ since we may take the identity subgroup. We assume $q>1$.
    Without loss of generality, assume that each $\gamma_l \in \on{Mat}_{q\times 1}(\Z)$ and acts in $\T^1$ by multiplication on column vectors (a $1\times 1$ column vector). We use $\gamma_l^T$ to denote the transpose matrix in $\on{Mat}_{1 \times q}(\Z)$.
    If $(\gamma_l)_{l \in \N}$ mixes $\mu$ to $\nu$, we are done by Proposition \ref{prop: mixing implies density}. In particular, we get something much stronger than the desired conclusion. That is, for $\mu$ almost every $z$, the set $\{\gamma_l(z) \in \T^q: l \in \N\}$ is dense.
    We proceed, assuming that $(\gamma_l)_{l \in \N}$ fails to mix $\mu$ to $\nu$. 

    Note that it is enough to prove the statement for a subsequence of $\ga_l$. Moreover, if we replace $\ga_l$ by $\ga \ga_l$ for a fixed $\ga\in \SL_q(\bZ)$, then it is enough to prove the statement for the new sequence 
$\ga\ga_l$. Thus, along the proof we will take subsequences and replace $\ga_l$ by appropriate linear images of $\ga_l$ and by abuse of notation, allow ourselves to continue and denote the new sequence by $\ga_l$.
    
    By the bilinearity of formula \eqref{eq: mixing}, the failure of mixing means that there are integers $a \in \Z$ and $b \in \Z^q$ such that, if we denote by $\chi_a$ and $\chi_b$ the respective characters on $\T^1$ and $\T^q$, we have
    \begin{equation*}
        \lim_{l \to \infty} \int_{\T^1} \chi_b(\gamma_l(z)) \chi_a(z)d \mu(z) \neq \left(\int_{\T^1} \chi_a d\mu\right) \left(\int_{\T^q} \chi_b d\nu\right).
    \end{equation*}
    In other words, we have
    \begin{equation*}
        \lim_{l \to \infty} \int_{\T^1} \exp(2i\pi(a + \gamma_l^T(b)| z)) d\mu(z) \neq \left(\int_{\T^1} \chi_a d\mu\right) \left(\int_{\T^q} \chi_b d\nu\right).
    \end{equation*}
    Thus, $b$ is nonzero.
    As the right hand side is $0$ in this case, the left hand side must equal 1 for infinitely many $l$'s, so we may assume after passing to a subsequence that 
    \begin{equation}\label{eq: gammanT -a = b}
        \gamma_l^T(b) = -a \text{ for all } l \in \N.
    \end{equation}
    Without loss of generality, we can assume $b$ is primitive. 
    Let $\gamma \in \SL_{q}(\Z)$ be such that 
    \begin{equation}\label{eq: gammaY eq = b}
        \gamma^T(\ve_q) = b,
    \end{equation}
    where $\ve_q\in \bZ^q$ is the $q$'th vector in the standard basis. 
    On replacing $\ga_l$ by $\ga\ga_l$ we see that we may assume without loss of generality that $b=\ve_q$.  By \eqref{eq: gammanT -a = b} we deduce that
    \begin{equation}\label{eq: gan coor}
    \ga_l = \br{
    \begin{array}{c}
    \ga_l'\\ -a 
    \end{array}
    }
    \end{equation}
    for all $l$, where $\ga_l'\in \on{Mat}_{(q-1)\times 1}$. 

    Now we view $\ga_l':\bT^1\to \bT^{q-1}$ as a sequence  of homomorphisms
    and let $\nu'$ denote the Haar probability measure on $\bT^{q-1}$.
    Note that the weak* convergence $(\ga_l')_*\mu\to \nu'$ follows from the assumed convergence 
    $(\ga_l)_*\mu\to \nu$ by composition with the projection $\bT^q\to \bT^{q-1}$ that forgets the last coordinate. This relies on Lemma~\ref{lem: pushforward is continuous}. 
    \begin{claim}\label{claim: ga' mix}
    The sequence $\ga_l'$ must mix $\mu$ to $\nu'$. 
    \end{claim}
    \begin{proof}
        If not, by the same logic as above, there must be some $\ga'\in \SL_{q-1}(\bZ)$ and $a'\in \bZ$ such that on replacing $\ga'_l$ by $\ga'\ga_l'$ we have that 
    $$\ga_l' = \br{
    \begin{array}{c}
    \ga_l''\\ -a' 
    \end{array}
    }$$ This means that after an appropriate replacement of $\ga_l$ we have
    \begin{equation}\label{eq: two last coordinates}
    \ga_l =\br{
    \begin{array}{c}
    \ga''_l\\ -a'\\-a
    \end{array}
    }.
    \end{equation}
    Let $\del\in\SL_2(\bZ)$ be such that $\del\smallmat{a'\\a} = \smallmat{*\\0}$ and let $\ga\in \SL_q(\bZ)$ be equal to the identity matrix with the bottom right $2\times 2$ block replaced by $\del$. Then
    by the choice of $\del$ and ~\eqref{eq: two last coordinates} we have that the last coordinate of $\ga \ga_l$ is $0$. Replacing $\ga_l$ by $\ga\ga_l$ we see that the image of the homomorphism $\ga_l:\bT^1\to \bT^q$ is contained in the subtorus defined by requiring the last coordinate to vanish. This contradicts the weak* convergence $(\ga_l)_*\mu\to \nu$ and so finishes the proof of the claim.
    \end{proof}
    We apply Proposition~\ref{prop: mixing implies density} to the sequence of homomorphisms $\ga_l':\bT^1\to\bT^{q-1}$ and conclude that for $\mu$-almost any $y\in \bT^1$, the sequence of images $\seta{\ga_l'y \in \bT^{q-1}}$ is dense. For each such $y$ we have by \eqref{eq: gan coor} that 
   $$\seta{\ga_ly:l\in \bN} = \seta{\mat{\ga_l'y\\-ay}\in \bT^q:l\in \bN}$$
    and so $\overline{\seta{\ga_ly:l\in \bN}}$ contains a coset (depending on $y$) of the subtorus obtained by requiring the last coordinate to vanish.
\end{proof} 

\section{Proof of Theorem \ref{thm: main theorem F formulation}}\label{sec: long sequences}
\ignore{
We can now give proof of Theorem \ref{thm: main theorem F formulation}. We first need a little more notation.
Specifically, it will now be very convenient to consider our algebraic measures on tori as lying in fibers in $Y$ over points in $X$:
Let $x \in X$. We consider the fiber $\pi^{-1}(x) \subset Y$ coming from the map in \eqref{diag: equivariant-projection}.
We have
\begin{equation*}
    \pi^{-1}(x) = \left\lbrace (g,v) \Gamma\ltimes \Z^{m+1}: g\Z^{m+1} = x,\ v \in \R^{m+1}\right\rbrace.
\end{equation*}
For $g \in G$ with $x = g\Z^{m+1}$, the map
\begin{equation*}
    \R^{m+1} \to \pi^{-1}(x);\ v \mapsto (g,v)\Gamma\ltimes \Z^{m+1}
\end{equation*}
induces a diffeomorphism 
\begin{equation}\label{eq: canonical diffeomorphism}
    \R^{m+1}/x \to \pi^{-1}(x).
\end{equation} 
Moreover, this diffeomorphism is canonical in the sense that it does not depend upon the choice of representative $g$.
\mycomment{The terminology in this definition is not good. Algebraic measures on homogeneous spaces is a taken term. You can call these algebraic measures on fibers.}
\begin{defn}[Algebraic measures on {fibers of $Y$}]\label{defn: algebraic measures on fibers}
    A measure $\mu \in \mathcal{P}(Y)$ is said to be an \new{algebraic on a fiber of $Y$} if there is some $x \in X$ such that $\mu$ arises as the pushforward of an algebraic measure on $\R^{m+1}/x$ by the canonical diffeomorphism of \eqref{eq: canonical diffeomorphism} followed by the inclusion $\pi^{-1}(x) \to Y$. That is, an algebraic measure on $Y$ comes from an algebraic measure on some $\R^{m+1}/x$ via the maps
    \begin{equation*}
        \R^{m+1}/x \to \pi^{-1}(x) \to Y.
    \end{equation*}
\end{defn}
\begin{proof}[Proof of Theorem \ref{thm: Uri F formulation}]
Let $x_0 \in X$ and let $x_1 \in \partial(x_0)$. Let $\mu \in \ca{P}(\R^{m+1}/x_0)$ be the full Haar probability measure. 
Since $x_1\ \in \partial(x_0)$, we have a divergent sequence of positive times $(t_l)_{l \in \N}$ and a sequence $(\varepsilon_l)_{l \in \N} \subset G$ converging to the identity such that
\begin{equation*}
    \varepsilon_l h_{t_l} x_0 = x_1 \text{ for all } n \in \N.
\end{equation*}
This equation gives us commutative diagrams:
\begin{equation}\label{diag: connecting fibers Uri}
        \begin{tikzcd}
  \R^{m+1}/x_0   \arrow[r, "\varepsilon_l h_{t_l}"] \arrow[d]  & \R^{m+1}/ x_1   \arrow[d]  \\
  \pi^{-1}(x_0) \arrow[r, "(\varepsilon_l h_{t_l}{,}0)"]  \arrow[d] & \pi^{-1}(x_1)\arrow[d] \\
  Y \arrow[r, "(\varepsilon_l h_{t_l}{,}0)"] & Y
\end{tikzcd}
\end{equation}
where the top vertical maps are the canonical diffeomorphisms described in \eqref{eq: canonical diffeomorphism}, the lower vertical maps are inclusions, the top horizontal map is given by $v + x_0 \mapsto \varepsilon_l h_{t_l}v + x_1$, and the lower  two horizontal  maps are given by the group action on $Y$.
We also consider $\mu$ to be a measure on $Y$ by pushforward via the left vertical arrows in diagram \eqref{diag: connecting fibers Uri}, with context making clear which set we are measuring in. 
An application of Proposition \ref{prop: Haar mixing} with the commutative diagram \ref{diag: connecting fibers Uri} shows that for $\mu$ almost every $y \in Y$, we have 
\begin{equation*}
    \overline{\left\lbrace (\varepsilon_l h_{t_l},0) y \in \pi^{-1}(x_1) : n \in \N\right\rbrace} \supset \pi^{-1}(x_1).
\end{equation*}
But the sequence $(\varepsilon_l,0)_{l \in \N} \subset G\ltimes \R^{m+1}$ converges to the identity and thus we even have
\begin{equation*}
    \overline{\left\lbrace ( h_{t_l},0) y \in \pi^{-1}(x_1) : n \in \N\right\rbrace} \supset \pi^{-1}(x_1).
\end{equation*}
An application of Lemma \ref{lem: contains values of limit} shows that and use of the diagram \ref{diag: connecting fibers Uri} shows that, for $\mu$ almost every  $v + x_0 \in \R^{m+1}/x_0$, we have
\begin{equation*}
\begin{split}
    \overline{F((x_0 + v) \cap \R^{m+1}_+)} &\supset \bigcup_{w \in \R^{m+1}} F((x_1 + w) \cap \R^{m+1}_+) \\
    &= F(\R^{m+1}_+) = [0,\infty).
\end{split}
\end{equation*}
Since $\mu$ comes from the Lebesgue measure on a fundamental domain for $x_0$, we obtain the result.
\end{proof}
}
Here is a sketch of the proof of Theorem~\ref{thm: main theorem F formulation} that should motivate the technicalities to come. 
The proof uses as an input
an accumulation sequence $(x_0,x_1,\dots, x_{d-1})$. We start with a one-dimensional algebraic measure $\mu$ on the torus $\bR^d/x_0$ and track what happens to the measure $\mu$ when we push it using sequences $h_{t_l}$ witnessing 
that $x_1\in \partial(x_0),x_2\in \partial(x_1)$ etc. At each step, we may assume the measures converge and we move from an algebraic measure on $\bR^d/x_i$ to an algebraic measure on $\bR^d/x_{i+1}$. The requirement 
regarding the length of the accumulation sequence is there to ensure that at the last step we must see the Haar probability measure on $\bR^d/x_{d-1}.$
The proof ends by invoking the Coset Lemma~\ref{lem: coset} and using the Inheritance Lemma~\ref{lem: inheritance} together with the fact that the non-degeneracy degree of $F$ is $d-1$, which is Lemma~\ref{lem: non degeneracy degree}. 

In order for this strategy to work we must know that the dimension of the algebraic measure we obtain in the limit at each stage jumps by at least 1. This is  the goal of of the next section and the reason for our assumption 
$\gcd(m,n)=1$.
\subsection{Jump in dimension}
    Consider the Euclidean space $\R^d$ with the standard inner product. Fix a dimension $1\le k\le d$ and consider the $k$'th exterior power
    $
        \wedge^k \R^d. 
    $
    On this vector space, we induce an inner product by declaring the standard basis vectors $\ve_I$ orthogonal. Here $I$ is a multi-index of length $k$, that is,
    a sequence $1\le i_1 < \dots < i_k \le d$ and 
    $\ve_I:=\ve_{i_1}\wedge \dots \wedge \ve_{i_k}.$ We denote by $\cI_k$ the set of all multi-indices of length $k$.
    
    The following Lemma is left as an exercise. It explains how the exterior power relates to our discussion: We deal with algebraic measures supported to a subtorus of $\bR^d/x$ where $x\in X$ is a lattice. Such a subtorus corresponds to an $x$-rational subspace $V$ and to a vector in 
    the appropriate exterior power obtained by wedging a $\bZ$-basis of $V\cap x$.
    \begin{lem}
        \label{lem: norm of pure tensors}
        Let $\mb{u}=u_1\wedge\dots\wedge u_k\in \wedge^k\bR^d$. Then, $\norm{u}$ is the $k$-dimensional volume of the parallelepiped spanned 
        by $u_1,\dots, u_k$ in $\bR^d.$ In particular, if $V<\bR^d$ is a $k$-dimensional subspace and $v_1,\dots, v_k$,$w_1,\dots ,w_k$ are $\bZ$-bases of a lattice in $V$, then $v_1\wedge\dots \wedge v_k = \pm w_1\wedge\dots \wedge w_k$.  
    \end{lem}
    Each linear map $g: \R^d \to \R^d$ induces uniquely a linear map
    \begin{equation*}
        \wedge g : \wedge^k \R^d \to \wedge^k \R^d
    \end{equation*}
    characterized by the way it acts on pure tensors: For all $v_1,\dots,v_k\in \bR^d$, $\wedge g(v_1\wedge\dots \wedge v_k) = (gv_1)\wedge\dots\wedge(gv_k)$. We consider the semigroup of matrices $(h_t)_{t \in \R_{\geq 0}} \in G$ as in \eqref{eq: ht definition} and consider their linear action on $\R^d$ by multiplication on column vectors.
    These induce the family of linear maps 
    \begin{equation*}
        (\wedge h_t)_{t \in \R_{\geq 0}} : \wedge^k \R^d \to \wedge^k \R^d
    \end{equation*}
    which are simultaneously diagonable. A basis of orthogonal eigenvectors is given by the standard basis vectors $\seta{\ve_I:I\in \cI_k}$ as we will show below.
    We split the indices $\br{d}:=\seta{1,\dots,d}$ into $J_+:=\seta{1,\dots, m}$ and $J_-:=\br{d}\sm J_+.$ With this notation it is easy to see that for any multi-index $I\in \cI_k$,
    \begin{equation}\label{eq: the weights}
    (\wedge h_t)\ve_I = \exp\pa{(\av{I\cap J_+}n - \av{I\cap J_-}m)t}\cdot \ve_I.
    \end{equation}
    We denote 
    \begin{equation}
        \label{eq: the weights2}
        \al_I(t):=\exp\pa{(\av{I\cap J_+}n - \av{I\cap J_-}m)t}
    \end{equation}
    and note that if the exponent $(\av{I\cap J_+}n - \av{I\cap J_-}m)$ is not zero, then $\al_I(t)\to \infty$ or $\al_I(t)\to 0$ as $t\to\infty$. 
\begin{lem}\label{lemma: zero is not a weight}
    Assume $\gcd(n,m)=1$. Then, for any $1\le k<d$, there is no $I\in \cI_k$ 
    such that $\al_I(t)\equiv 1$. Moreover, 
    for any $u\in \wedge^k \bR^d$ we have that $\norm{(\wedge h_t)u}$ either diverges or converges to $0$ as $t\to\infty.$
\end{lem}
\begin{proof}
    Fix $k$ as in the statement and let $I\in \cI_k$. By equation~\eqref{eq: the weights2}, if 
     $\al_I(t)\equiv 1$, then there are $1\le a \le b\le k$ such that $a+b=k$ and $an-mb=0$. By the $\gcd$ assumption we deduce that $m|a$ and $n|b$. But this gives that  $d=  m+n\le a+b=k<d $ which is absurd.

    Finally, if $u\in \wedge^k \bR^d$ we may write $u=\sum_{I\in \cI_k} a_I\ve_I$. By \eqref{eq: the weights} we have that 
    $$(\wedge h_t)u = \sum_{I\in \cI_k} a_I (\wedge h_t)\ve_I = 
    \sum_{I\in\cI_k} \al_I(t)a_I \ve_I.$$
    By the first part, for each $I$, $\lim_{t\to\infty} \al_I(t)\in \seta{0,\infty}$ and hence we deduce that $\norm{(\wedge h_t)u}\to \infty$  unless $a_I=0$ for every $I$ for which $\al_I(t)\to \infty$, in which case $\norm{(\wedge h_t)u}\to 0$.
    
\end{proof}
\begin{cor}\label{cor: subspaces do not stabilize}
    Assume $\gcd(m,n)=1$. 
    Suppose $x_0,x_1\in X$ are lattices such that $x_1\in \partial(x_0)$. Let
    $t_l\to\infty$ and $G\ni \varepsilon_l\to e$ be such that $\varepsilon_l h_{t_l}x_0=x_1.$ Then, if
    $U<\bR^d$ is a $k$-dimensional
    $x_0$-rational subspace with $1\le k <d$, it is not possible that the sequence of subspaces $\varepsilon_l h_{t_l} U$ stabilizes. 
\end{cor}
    \begin{proof}
        Suppose on the contrary that the $x_1$-rational subspace 
        $\varepsilon_l h_{t_l} U$ is independent of $l$ and denote it by $V$. 
        Note that $\varepsilon_lh_{t_l}(U\cap x_0) = V\cap x_1$. 
        Choose a basis
        $u_1,\dots, u_k$ of $U\cap x_0$ and denote $\mb{u}:=u_1\wedge\dots \wedge u_k\in \wedge^k\bR^d$. We have that for any $l$ 
        $$(\wedge \varepsilon_l)(\wedge h_{t_l}) (\mb{u}) = \wedge(\varepsilon_l h_{t_l})(\mb{u}) =(\varepsilon_l h_{t_l} u_1)\wedge\dots \wedge(\varepsilon_l h_{t_l}u_k). $$
        The right hand side is a pure tensor obtained by wedging a basis of the lattice $x_1\cap V$ in $V$ and hence, by Lemma~\ref{lem: norm of pure tensors} does not depend on $l$ (up to a sign maybe). In particular, the norm of this vector is bounded
        above and below (away from $0$). On the other hand, the left hand side goes to infinity or to zero by Lemma~\ref{lemma: zero is not a weight} because the operators $\wedge \varepsilon_l$ converge to the identity map on $\wedge^k\bR^d$. This gives the desired contradiction.
    \end{proof}

\begin{thm}[Jump in dimension]\label{thm: ht limits must jump}
    Assume $\gcd(m,n) =1$. Let $x_0, x_1 \in X$ with $x_1 \in \partial(x_0)$. Let $\mu$ be an algebraic measure on $\bR^d/x_0$ of dimension $1\le k<d$.
    Then, there is a divergent sequence $(t_l)_{l \in \N} \subset \R_{\geq 0}$ with
    \begin{equation*}
        \lim_{l \to \infty} (h_{t_l})_* \mu = \nu \ \text{ in the weak* topology of } \mathcal{M}(Y),
    \end{equation*}
    where $\nu$ is an algebraic measure on $\bR^d/x_1$ and moreover, 
    \begin{equation}\label{eq: dimension jump by 1}
        \dim \nu \geq k+1.
    \end{equation}
\end{thm}

\begin{proof}
    Choose a divergent sequence $(t_l)_{l \in \N} \subset \R_{\geq 0}$ witnessing $x_1 \in \partial(x_0)$. 
    Let $(\varepsilon_l)_{l \in \N} \subset G$ be a sequence converging to the identity with
    \begin{equation*}
        \varepsilon_l h_{t_l} x_0 = x_1 \ \text{ for all }\ l \in \N.
    \end{equation*}
Acting on $Y$ rather than on $X$ the above equation says that $\varepsilon_l h_{t_l}$ maps $\pi^{-1}(x_0) = \bR^d/x_0$ to $\pi^{-1}(x_1)=\bR^d/x_1$ and is in fact a group homomorphism between these two tori.
By Lemma \ref{lem: B-A}, we can assume that we have a weak* limit
\begin{equation}\label{eq: limit exists}
   \lim_{l \to \infty} (h_{t_l})_* \mu= \lim_{l \to \infty} (\varepsilon_lh_{t_l})_* \mu = \nu \in \cP(\pi^{-1}(x_1)),
\end{equation}
where the equality on the left follows by applying Lemma \ref{lem: weak limit is the same under perturbations} and noting that $(\varepsilon_l)_{l \in \N}$ converges to the identity in $G$. An application of  Corollary 
\ref{cor: limit of algebraic measures is algebraic} shows that $\nu$ must be algebraic.

It remains to establish \eqref{eq: dimension jump by 1}. Assume that the algebraic measure $\mu$ on $\R^d/x_0$ corresponds to the $x_0$-rational subspace $U \subset \R^d$.
Then, the measures $(\varepsilon_lh_{t_l})_*\mu\in \cP(\bR^d/x_1)$ are algebraic measures that correspond to the $x_1$-rational subspaces $V_l: =\varepsilon_lh_{t_l}U$ and are thus of a fixed dimension $k$. Applying 
Corollary \ref{cor: limit of algebraic measures is algebraic} we deduce that 
$\dim \nu\ge k$ and that if $\dim \nu = k$, then along a subsequence, $\varepsilon_lh_{t_l}U$ is a fixed $x_1$-rational subspace. 
This contradicts Corollary~\ref{cor: subspaces do not stabilize} because we assume $\gcd(m,n)=1$ and $1\le k<d$.
\end{proof}
The following Corollary bootstraps Theorem~\ref{thm: ht limits must jump}
for a long accumulation sequence. This is used in order to reach the dimension of the non-degeneracy degree of $F$ as appears in Lemma~\ref{lem: non degeneracy degree}.
\begin{cor}\label{cor: acc seq measure dimension}
   Assume $\gcd(m,n) =1$. Let $x_0 \in X$ be a lattice with an accumulation sequence 
$        (x_0, \dots, x_r) \in X^{r+1}.$
    Let $\mu$ be an algebraic measure on $\bR^d/x_0$ of dimension $1\le k<d$.
    Then, there is a divergent sequence $(t_l)_{l \in \N} \subset \R_{\geq 0}$ with
    \begin{equation*}
        \lim_{l \to \infty} (h_{t_l})_* \mu = \nu \ \text{ in the weak* topology of } \mathcal{M}(Y),
    \end{equation*}
    where $\nu$ is an algebraic measure on $\bR^d/x_r$ and moreover, 
    \begin{equation}\label{eq: dimension jump by more}
        \dim \nu \geq \min\seta{k+r, d}.
    \end{equation}
\end{cor}
\begin{proof}
    We prove this by induction on the length of the accumulation sequence $r$. When $r=1$ this is Theorem~\ref{thm: ht limits must jump}. Assume the statement holds for accumulation sequences of length $r-1\ge1$ and let $x_0,\dots, x_r$, $\mu$ be as in the statement. By the induction 
    hypothesis there exists a sequence $t_l\to\infty$ such that $(h_{t_l})_*\mu\to \eta$, where $\eta$ is an algebraic measure in $\cP(\bR^d/x_{r-1})$ of dimension $\ge \min\seta{k+r-1,d}$. Applying 
    Theorem~\ref{thm: ht limits must jump} to $x_{r-1},x_r$ we deduce the existence of a sequence $s_l\to \infty$ such that $(h_{s_l})_*\eta\to \nu$ where $\nu\in\cP(\bR^d/x_r)$ is an algebraic measure with $\dim \nu\ge \min\seta{k+r,d}$. An application of 
    Lemma~\ref{lem: convergence transitive} gives the existence of sequences $p_l,q_l$ of natural numbers such that
    $$h_{s_{p_l}+t_{q_l}}=(h_{s_{p_l}})_*(h_{t_{q_l}})*\mu\to \nu$$
    which finishes the proof.

\end{proof}

We now come to the
\begin{proof}[Proof of Theorem \ref{thm: main theorem F formulation}]
    Let $x \in X$ be as in the statement and
    \begin{equation*}
        (x = x_0, x_1, \dots, x_{d-1}) \in X^d
    \end{equation*}
    be an accumulation sequence of length $d$ for $x$. 
    Let $\mu$ on $\pi^{-1}(x)$ be an algebraic measure. Similarly to the argument showing the equivalence between the formulations of Theorems \ref{thm: 1 dim in Zm-2} and \ref{thm: 1 dim in Zm}, we may assume that $\mu$ is 1-dimensional.
    Applying Corollary \ref{cor: acc seq measure dimension}, we see that there is a sequence of times $(t_l)_{l \in \N} \subset \R_{\geq 0}$ and an algebraic measure $\nu$  on $\pi^{-1}(x_{d-1})$ such that 
    $\lim_{l \to \infty} (h_{t_l})_* \mu = \nu.$
    Moreover, because of the inequality \eqref{eq: dimension jump by more}, we must have $     \dim \nu= d$
    so that in fact, $\nu$ is the Haar probability measure on $\R^d/x_{d-1}$. Let $(\varepsilon_l)_{l \in\N} \subset G$ be a sequence converging to the identity with 
    \begin{equation*}
        \varepsilon_l h_{t_l} x = x_{d-1} \text{ for all } l \in \N.
    \end{equation*}
    Acting on $Y$ rather on $X$, the above equation means that acting with
    $\varepsilon_l h_{t_l}$ maps the fiber $\pi^{-1}(x_0)=\bR^d/x_0$ to 
    $\pi^{-1}(x_{d-1})=\bR^d/x_{d-1}$. 
Applying the Coset Lemma \ref{lem: coset}, we see that for $\mu$ almost every $y \in Y$, the closure
$\overline{\left\lbrace \varepsilon_l h_{t_l}y \in Y: l \in \N \right\rbrace}$
contains a coset of a subtorus of $\bR^d/x_{d-1}$ of dimension $d-1$. 
Since $(\varepsilon_l)_{l \in \N}$ conoverges to the identity, the same is true for the set
$   \overline{\left\lbrace h_{t_l}y \in Y: l \in \N \right\rbrace}$
where $y$ lies in a set of full $\mu$ measure.
Proposition \ref{prop: codimension 1 subtorus} implies that each such grid $y$
is $DV_F$
and we are done.
\end{proof}

\section{Dual approximation, transference theorems and a counterexample}\label{sec: transference}

In order to prove Theorem \ref{thm: counterexample intro}, we switch to the dual setting of approximation of linear forms to make use of some powerful transference theorems due to Jarnik and Khintchine (See Theorem \ref{thm: Jarnik transference}). We need some notation.

\begin{defn}[Irrationality measure function]
    Let $\theta = (\theta_1, \theta_2, \dots ,\theta_m) \in \R^m$. 
    We define $\Psi : [1,\infty) \to  [0,\infty)$ by
    \begin{equation*}
        \Psi(t) = \min\left\lbrace  \langle a_1 \theta_1 + \dots + a_m \theta_m\rangle : a_i \in \Z \text{ and } 0 < |a_i| \leq t \text{ for }  i=1,\dots m \right\rbrace.
    \end{equation*}
    This function is piecewise constant and nonincreasing. Its points of discontinuity occur at integer values. We denote these integers by the increasing sequence $(M_l)_{l \in \N} \subset \N$ and write the images as
    \begin{equation*}
        \zeta_l := \Psi(M_l) \text{ for all } l \in \N.
    \end{equation*}
\end{defn}

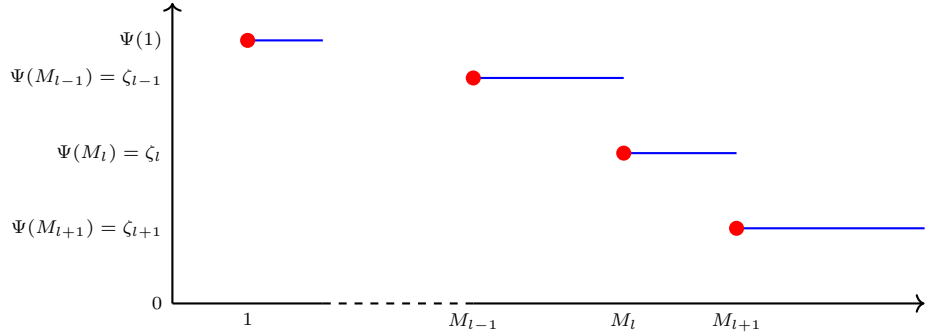
\begin{figure}[!h]
    \begin{tikzpicture}[scale=1, baseline=(current  bounding  box.center)]
    
    \draw [thick, ->] (-5,-2) -- (-5,2);
    \draw [thick] (-5,-2) -- (-3,-2);
    \draw [thick, ->] (-1,-2) -- (5,-2);

    \draw[thick, dashed] (-3,-2)--(-1,-2);

    \node [left] at (-5,-2) {\tiny $0$};
    
    \node [below] at (-4,-2) {\tiny $1$};

    \draw[thick, blue] (-4,1.5) -- (-3,1.5);
    \fill [red] (-4,1.5) circle[radius=0.1];
    \node [left] at (-5,1.5) {\tiny $\Psi(1)$};

    \node [below] at (-1,-2) {\tiny $M_{l-1}$};

    \draw[thick, blue] (-1,1) -- (1,1);
    \fill [red] (-1,1) circle[radius=0.1];
    \node [left] at (-5,1) {\tiny $\Psi(M_{l-1}) = \zeta_{l-1}$};

    \draw[thick, blue] (1,0) -- (2.5,0);
    \fill [red] (1,0) circle[radius=0.1];
    \node [left] at (-5,0) {\tiny $\Psi(M_l) = \zeta_{l}$};
    
    \draw[thick, blue] (2.5,-1) -- (5,-1);
    \fill [red] (2.5,-1) circle[radius=0.1];
    \node [left] at (-5,-1) {\tiny $\Psi(M_{l+1}) = \zeta_{l+1}$};
    
    \node [below] at (1,-2) {\tiny $M_l$};
    
    \node [below] at (2.5,-2) {\tiny $M_{l+1}$};

    \end{tikzpicture}
    \caption{Here is a graph of $\Psi$ to remember the notation.}
\end{figure}

We have the following basic estimate.
\begin{lem}
    For every $\theta \in \R^m$ and associated irrationality measure function $\Psi$, we have
    \begin{equation*}
        \Psi(t) \leq t^{-m}.
    \end{equation*}
\end{lem}
\begin{proof}
     This follows from Minkowski's convex body theorem \cite[II, Theorem 2B]{Schmidt-DA}.
\end{proof}
A key fact we use is the following.
 \begin{prop}\label{prop: special singular form}
     Let $m$ be an integer with $m > 2$. There exists $\theta \in \mathbb{R}^2$ with irrationality measure function $\Psi$ satisfying, for some positive constants $C_1$ and $C_2$, the inequalities
     \begin{equation*}
         \Psi(t) t^m \leq C_1 \text{ for all }  t \geq 1,
     \end{equation*}
     and
     \begin{equation*}
         C_2 \leq \Psi(t) t^m \text{ for an unbounded set of } t \geq 1. 
     \end{equation*}
 \end{prop}
 \begin{rem}
     Since the proof of this uses the parametric geometry of numbers as in the paper \cite{Roy-MathZ} and requires additional notation, its proof is deferred to the appendix.
 \end{rem}

The conditions above are related to the Diophantine properties of the (column) vector $\theta \in \R^2$ by the following transference theorem.
\begin{thm}[Theorem 7 in \cite{Jarniktransference}]\label{thm: Jarnik transference}
    Suppose $\theta \in \R^2$ with associated irrationality measure function $\Psi$. Suppose further, we have an integer $m >2$ and a constant $C_1 > 0$ with
    \begin{equation*}
        \Psi(t) t^m \leq C_1 \text{ for all } t \geq 1.
    \end{equation*}
    Then, there exists $\eta = (\eta_1, \eta_2) \in \R^2$ with 
    \begin{equation*}
        \inf_{q \in \Z \smallsetminus \{0\}}\ |q|^{1/m} \langle q\theta - \eta\rangle > 0.
    \end{equation*}
\end{thm}
Next, we see that the second property in \ref{prop: special singular form} is useful in extending the components of the (column) vector $\theta \in \R^2$ to obtain a nonsingular one in higher dimensions. 
\begin{thm}\label{thm: augment to get nonsingular}
    Assume $\theta = (\theta_1, \theta_2) \in \R^2$ with irrationality measure function $\Psi$. Suppose $m$ is an integer with $m > 2$, and we have a constant $C_2 >0$ with
    \begin{equation}\label{eq: lower bound Psi}
        C_2 \leq \Psi(t) t^m \ \text{ for an unbounded set of } t \geq 1.
    \end{equation}
    Then, for Lebesgue almost every $(\theta_3,\dots, \theta_m) \in \R^{m-2}$, the vector $(\theta_1, \theta_2, \theta_3,\dots,\theta_m) \in \R^m$ is nonsingular. 
\end{thm}
\begin{proof}
    First, recall the sequence of integers $(M_l)_{l \in \N}$ associated to the irrationality measure function $\Psi$.
    The condition \eqref{eq: lower bound Psi} means there is an infinite subset $I \subset \N$ such that, for each $l \in I$, we have
    \begin{equation}\label{eq: nonsingular hypothesis Mn}
        C_2 \leq \Psi(M_{l}) M_{l+1}^m.
    \end{equation}
    Next, we need some auxiliary claims:
    \begin{claim}\label{claim: aux 1}
        Let $0 < \delta < 1/2$. Given a nonzero integer vector $(a_3,\dots, a_m) \in \Z^{m-2}$ and $M \in \N$, let $S(M, \delta,a_3,\dots, a_m)$ denote the set of integer vectors $(a_0, a_1, a_2) \in \Z^3$ for which
        \begin{equation*}
            \max_{i=1,2} |a_i| < M
        \end{equation*}
        and for which there exists $(\theta_3, \dots, \theta_m) \in [0,1]^{m-2}$ with 
        \begin{equation*}
            |a_0 +  a_1 \theta_1 + a_2 \theta_2 + a_3 \theta_3 + \dots + a_m \theta_m| \leq \delta.
        \end{equation*}
        Then 
        \begin{equation*}
            \# S(M, \delta,a_3, \dots, a_m) \leq 18 M^2 (|a_3| + \dots |a_m| + 1).
        \end{equation*}
    \end{claim}
    \begin{proof}[Proof of Claim \ref{claim: aux 1}] 
        Say $(a_0, a_1, a_2) \in S(M, \delta, a_3,\dots, a_m)$.
        We see that $a_0$ must certainly satisfy
        \begin{equation*}
           - (a_1 \theta_1 + a_2 \theta_2)  - (|a_3| + \dots +|a_m|) - \delta \leq a_0 \leq 
            - (a_1 \theta_1 + a_2 \theta_2) + (|a_3| + \dots +|a_m|) + \delta.
        \end{equation*}
        Whence, for fixed $(a_1, a_2)$ there are at most
        \begin{equation*}
            2(|a_3| + \dots |a_m|+1)
        \end{equation*} 
        options for $a_0$. Moreover, there are $2M+1 \leq 3M$ options for each of $a_1$ and $a_2$. This gives us the required bound.
    \end{proof}

For an integer vector $(a_0, a_1, \dots, a_m) \in \Z^{m+1}$ with $(a_3,\dots, a_m)$ nonzero, $M \in \N$ and $ 0<\varepsilon < 1/2$, we define the open subset $B_{M, \varepsilon}(a_0, \dots, a_m) \subset [0,1]^{m-2}$ as:
\begin{equation*}
    \left\lbrace (\theta_3, \dots , \theta_m) \in [0,1]^{m-2}: |a_0 + a_1 \theta_1 + a_2 \theta_2 + a_3\theta_3 + \dots + a_m \theta_m| < \frac{\varepsilon}{M^m}\right\rbrace.
\end{equation*}
Further, we define $B_{M,\varepsilon}$ as the union
\begin{equation*}
     \bigcup \left\lbrace B_{M, \varepsilon}(a_0, \dots, a_m) : (a_0, \dots, a_m) \in \Z^{m+1},\ (a_3, \dots, a_m) \neq 0,\ \max_{i=1,\dots,m} |a_i| < M\right\rbrace.
\end{equation*}
We use $\lambda$ to  denote the Lebesgue measure on $\R^{m-2}$. 
Viewing the inequality occurring in the definition of $B_{M, \varepsilon}(a_0, \dots, a_m)$ in terms of an inner product with a unit vector multiple of $(a_3, \dots, a_m)$, we see that 
\begin{equation}\label{eq: B measure}
\lambda(B_{M, \varepsilon}(a_0, \dots, a_m)) \leq (m-2)^{(m-3)/2}\frac{2 \varepsilon}{M^m \sqrt{a_3^2+...+a_m^2}}.
  \end{equation}
  \begin{claim}\label{claim: BMe measure}
      For any $M \in \N$ and $\varepsilon$ with $0 < \varepsilon < 1/2$, we have
      \begin{equation*}
          \lambda(B_{M, \varepsilon}) < K \varepsilon
      \end{equation*}
      where $K>0$ is a constant depending only on $m$.
  \end{claim}
  \begin{proof}[Proof of Claim \ref{claim: BMe measure}]
      We have the union
      \begin{equation*}
          B_{M, \varepsilon} = \bigcup_{\substack{(a_3,\dots, a_m) \in \Z^{m-2} \smallsetminus \{0\},\\ |a_3| ,\dots, |a_m| < M}} \bigcup \left\lbrace B_{M, \varepsilon}(a_0,\dots, a_m): (a_0,a_1,a_2) \in S(M, \varepsilon/M^m, a_3,\dots, a_m) \right\rbrace.
      \end{equation*}
      Applying the estimates of Claim \ref{claim: aux 1} and \eqref{eq: B measure}, we get
      \begin{equation*}
          \lambda(B_{M, \varepsilon}) \leq \sum_{\substack{(a_3,\dots, a_m) \in \Z^{m-2} \smallsetminus \{0\},\\ |a_3| ,\dots, |a_m| < M}} 18M^2 (|a_3| + \dots |a_m| + 1) \cdot \frac{2\varepsilon(m-2)^{(m-3)/2} }{M^m \sqrt{a_3^2 + \dots + a_m^2}}.
      \end{equation*}
This gives us the claim.
\end{proof}
Consider the sets, for a fixed $0 <\varepsilon< 1/2$ and each $l \in I$,
\begin{equation*}
    W_l(\varepsilon) := [0,1]^{m-2} \smallsetminus B_{M_l, \varepsilon}.
\end{equation*}
Let $W(\varepsilon)$ be the limsup set
\begin{equation*}
    W(\varepsilon) := \bigcap_{p \in \N}\ \bigcup_{\substack{l > p\\ l \in I}} W_l(\varepsilon)
\end{equation*}
and let 
\begin{equation*}
    W := \bigcup_{0 < \varepsilon< 1/2} W(\varepsilon).
\end{equation*}
For each $l \in I$ and $0 < \varepsilon < 1/2$, we have from Claim \ref{claim: BMe measure} that
\begin{equation*}
    \lambda (W_l(\varepsilon)) \geq 1 - K \varepsilon.
\end{equation*}
Further, for $l_1, l_2 \in I$, we have the trivial bound
\begin{equation*}
    \lambda(W_{l_1}(\varepsilon) \cap W_{l_2}(\varepsilon)) \leq 1.
\end{equation*}
This gives us the estimate
\begin{equation*}
    \limsup_{t\to \infty} \frac{\left(\sum_{l=1, l \in I}^t \lambda(W_l(\varepsilon))\right)^2}{ \sum_{l_1, l_2 =1, l \in I}^t \lambda(W_{l_1}(\varepsilon) \cap W_{l_2}(\varepsilon))} 
    \geq (1 - K \varepsilon)^2.
\end{equation*}
It then follows from \cite[Theorem 18.10]{Gut} that, for every $0 < \varepsilon < 1/2$,
\begin{equation*}
    \lambda(W(\varepsilon)) \geq (1- K\varepsilon)^2.
\end{equation*}
In particular, we have $\lambda(W) = 1$.
Now let $(\theta_3,\dots, \theta_m) \in W$ and consider the vector $(\theta_1, \dots, \theta_m) \in \R^m$.
There is some $0 < \varepsilon < 1/2$ for which $(\theta_3, \dots, \theta_m) \in W(\varepsilon)$.
In particular, we have an infinite number of $l \in I$ such that $(\theta_3, \dots, \theta_m) \notin B_{M_l, \varepsilon}$.

Fix such an $l > 1$. It follows from the definitions that, for each $(a_0, \dots, a_m) \in \Z^{m+1}$ with $(a_3, \dots, a_m) \neq 0$ and $\max_{i =1,\dots, m} |a_i| < M_l$, we have
\begin{equation*}
    |a_0 + a_1 \theta_1 + a_2 \theta_2 + a_3 \theta_3 + \dots + a_m \theta_m| \geq \frac{\varepsilon}{M_l^m}.
\end{equation*}
Further, if we have $(a_0, \dots, a_m) \in \Z^{m+1}$ with $\max_{i=1,\dots m} |a_i| < M_l$ and $(a_3, \dots, a_m) = 0$, then 
\begin{equation*}
\begin{split}
    |a_0 + a_1 \theta_1 + a_2 \theta_2 + a_3 \theta_3 + \dots + a_m \theta_m| &= |a_0 + a_1\theta_1 + a_2 \theta_2| \\
    &\geq \Psi(M_{l-1}) \\
    &\geq \frac{C_2}{M_l^m}
\end{split}
\end{equation*}
where we use the hypothesis \eqref{eq: nonsingular hypothesis Mn} in the last inequality.
If we put $t = M_l - \frac{1}{2}$ and if we let $\Phi$ denote the irrationality measure function of $(\theta_1, \dots, \theta_m) \in \R^m$, we get that
\begin{equation*}
    \Phi(t)t^m \geq \frac{1}{2^m} \min\{\varepsilon, C_3\} >0.
\end{equation*}

Allowing $l$ to vary, we see that $(\theta_1, \dots, \theta_m) \in \R^m$ is nonsingular.
\end{proof}
We finally come to our counterexample.
\begin{thm}\label{thm: counterexample}
    Let $m$ be an integer with $m>2$. Let $(\theta_1, \theta_2) \in \R^2$ be as in Proposition \ref{prop: special singular form}. 
    Then, there exists a $(\theta_3, \dots, \theta_m) \in \R^{m-2}$ and a $(\eta_1, \eta_2) \in \R^2$ such that
    \begin{enumerate}
        \item[(a)] The augmented vector 
        \begin{equation*}
        \theta = (\theta_1, \dots, \theta_m) \in \R^m
    \end{equation*}
    is nonsingular.
    \item[(b)] For any $(\eta_3, \dots, \eta_m) \in \R^{m-2}$, the augmented vector $\eta = (\eta_1, \dots, \eta_m) \in \R^m$ satisfies
    \begin{equation*}
        \inf_{q \in \Z \smallsetminus \{0\}} |q|^{1/m} \langle q\theta - \eta \rangle > 0.
    \end{equation*}
    \end{enumerate}
     In particular, we have a nonsingular $\theta \in \R^m$ with $\on{Bad}_\theta$ containing the translate of an $m-2$ dimensional subtorus in $\T^m$.
\end{thm}
\begin{proof}
    We apply Theorem \ref{thm: augment to get nonsingular} to obtain part (a). We apply Theorem \ref{thm: Jarnik transference} to obtain $(\eta_1, \eta_2) \in \R^2$. We then compute, for $q \in \Z$ nonzero,
    \begin{equation*}
    \begin{split}
        |q|^{1/m} \max_{i=1,\dots, m} \langle q\theta_i - \eta_i \rangle & \geq |q|^{1/m} \max_{i=1,2} \langle q\theta_i - \eta_i \rangle \\
        &> 0
    \end{split}
    \end{equation*}
    where we used the conclusion of Theorem \ref{thm: Jarnik transference} in the last inequality.
\end{proof}

\section{A Strengthening of Theorem~\ref{thm: Uri lebesgue}}
\label{sec: Kolya}

In this section we wish to define a Diophantine class of matrices 
in $\cC\subset \on{Mat}_{m\times n}(\bR)$ such that $\cC$ strictly contains the class of non-singular matrices and such that for any
$A\in \cC$, $\lam(\on{Bad}_A)=0$, where $\lam$ is the $m$-dimensional
Lebesgue measure on $\bT^m$. We begin by introducing the necessary terminology and notation for the definition of $\cC$. 

In this part of the paper we assume that $||\cdot || $ is the sup-norm,
that is for a vector
$z=(z_1,...,z_s)\in \mathbb{R}^s$ we have
$
||z|| = \max_{1\le i \le s} |z_i|$.
 \vskip+0.3cm

Recall that in our notation $ d = m+n$.
Suppose that the columns  $\theta^1,...,\theta^n$ of our 
   $m\times n $ matrix $$
   A = (\theta^1\cdots\theta^n)
\ignore{   
   =\left(
\begin{array}{ccc}
\theta_{1}^1&\cdots&\theta_{1}^n\cr
\theta_{2}^1&\cdots&\theta_{2}^n\cr
\cdots &\cdots &\cdots \cr
\theta_{m}^1&\cdots&\theta_{m}^n
\end{array}
\right)
}
   $$ are linearly independent over $\mathbb{Q}$ and consider the 
   irrationality measure function
 $$
 \Psi_A(t) = \min_{q\in \mathbb{Z}^n: 0\le ||q||\le t}  \langle Aq\rangle
$$
which is a piecewise  constant function decreasing to zero.
Let
$$
M_1<M_2<...<M_l<M_{l+1}<...
$$
be the infinite sequence of all the points where $\Psi_A(t) $ is not continuous and 
 $$
\zeta_l = \Psi_A(M_l), \,\,\,\,\,\, \zeta_l>\zeta_{l+1}
$$
be the corresponding sequence of values of the function.
In particular
$$
\zeta_l = \min_{q\in \mathbb{Z}^n: 0\le ||q||<M_{l+1}}  \langle Aq\rangle.
$$
The main properties of the function  
$\Psi_A(t)$ 
are discussed, for example in \cite{MoshchevitinSingular}.
From Minkowski convex body theorem it follows that 
\begin{equation}\label{mi1}
\Delta_l :=
M_{l+1}^n\zeta_l^m\le 1
\end{equation}
and
$A$ is singular if and only if $ \Delta_l \to 0$ as $ l \to \infty$.

 It is well known 
 (see for example
 \cite{BugeaudExponents})
 that the values of $M_l$ grow exponentially,
 that is
 \begin{equation}\label{bl}
M_{l+3^d+1} \ge 2M_l
\end{equation}
for all $l$. 

We now have enough notation and terminology in order to define the Diophantine class of matrices for which our result holds. 
\begin{defn}\label{defn: Kolya}
  Let $\cC\subset \on{Mat}_{m\times n}(\bR)$ be the collection of matrices $A$ such that 
  there exists an increasing sequence $l_k$ such that 
\begin{enumerate}
    \item\label{item: Kolya1} $\sum_{k=1}^\infty \Delta_{l_k}^{d-1} = \infty$;
    \item\label{item: Kolya2}   $
H_k:= \displaystyle{\sup_{k_1\ge k+1}}
\left(\frac{\zeta_{l_{k_1}}}{\Delta_{l_{k_1}}}\right)^m\cdot \left(  \frac{M_{l_{k}+1}}{\Delta_{l_{k}}}\right)^n \to 0, \,\,\ k \to \infty$.
\end{enumerate}
\end{defn}

\begin{prop}
    \label{prop: nonsingular implies C}
    If $A$ is a non-singular matrix then $A$ is in $\cC$.
\end{prop}
\begin{proof}
For non-singular  $A$ there exists a sequence $ l_k$ such that  
\begin{equation}\label{eq: strictly pos}
\inf_k \Delta_{l_k} >0.
\end{equation} 
So item \eqref{item: Kolya1} of Definition~\ref{defn: Kolya} holds 
for any subsequence of $l_k$.
It also follows from \eqref{eq: strictly pos} and monotonicity of 
$\zeta_l$ that we can choose a subsequence of $ l_k$ 
 to satisfy 
 $$H_k\ll \zeta_{l_{k+1}}^m M_{l_k+1}^n= 
 \pa{\frac{M_{l_k+1}}{M_{l_{k+1}+1}}}^n 
 \zeta_{l_{k+1}}^m M_{l_{k+1}+1}^n \le \pa{\frac{M_{l_k+1}}{M_{l_{k+1}+1}}}^n $$ 
 where we have used \eqref{mi1} in the last inequality. We can now take a subsequence of $l_k$ and use the exponential growth of the sequence $M_l$ to arrange that item \eqref{item: Kolya2} will hold.
 \end{proof}

The following proposition shows that the class $\cC$ strictly contains 
the class of non-singular matrices.
\begin{prop}\label{abcde}
There exists a singular matrix $A\in \cC$.
\end{prop}
\begin{proof}
It is well known, and may be easily deduced using the theory of templates,
that for any function $\rho(l)$ that satisfies $\lim_{l\to\infty}\rho(l)=0$, there exists a matrix $A$ with $\Del_l\asymp \rho(l)$. For our construction below, we choose $A$ such that 
$$\Del_l\asymp l^{-\frac{1}{2(d-1)}}.$$
The fact that $A$ is singular follows since $\Del_l\to 0$. We now show that the lower bound 
$\Delta_l>l^{-\frac{1}{2(d-1)}}$ implies that $A\in \cC$. 

Take $l_k = k^2$. Then, because of the lower bound we have that
$\sum_k \Delta_{l_k}^{d-1}$ diverges, so item \eqref{item: Kolya1} is satisfied. 
Moreover, the lower bound for $\Delta_l$ implies that for any
$k_1\ge k$ we have the inequality
$$
\frac{1}{
\Delta_{l_{k_1}}^n \Delta_{l_k}^m} \le k_1^{\frac{d}{d-1}},
$$
meanwhile
from (\ref{mi1}) and (\ref{bl}) it follows that 
$$
\zeta_{l_{k_1}}^m M_{l_k+1}^n \le
\frac{M_{k^2+1}^n}{M_{k_1^2+1}^n}\ll  2^{-\frac{n(k_1^2 - k^2)}{3^d+1}}.
$$
So 
\begin{align*}
H_k &\le \sup_{k_1\ge k+1} k_1^{\frac{d}{d-1}} \cdot 2^{-\frac{n(k_1^2 - k^2)}{3^d+1}}\\
&=
\sup_{j\ge 1} \, (k+j)^{\frac{d}{d-1}} \cdot 2^{-\frac{n(2kj+j^2)}{3^d+1}}\\
&\ll k^{\frac{d}{d-1}}\cdot 2^{-\frac{2nk}{3^d+1}} \to 0, \textrm{ as } k \to \infty
,
\end{align*}
and so item \eqref{item: Kolya2} is also valid.
\end{proof}
We now arrive to the main result of this section which strengthens
Theorem~\ref{thm: Uri lebesgue} and also shows that the fact that 
$\on{Bad}_A$ is a null-set with respect to the $m$-dimensional Lebesgue measure does not characterize non-singularity.

The proof of the following result is motivated by \cite[Theorem 1.7]{Kim-Kurzweil} and is based on an application of  Minkowski successive minima theory and modifies the proof from \cite{Mosh-welldistributed}.

\begin{thm}\label{thm: Kolya}
For any $A\in \cC$ one has
$\lambda_{\mathbb{T}^m} ({\rm Bad}_A) =0.
 $
\end{thm}

\begin{proof}
Let $A$ be any matrix and let $\Del_l, \zeta_l$ be as above.   
 For ${\eta} = (\eta_1,...,\eta_m) \in \mathbb{R}^m$ let us denote
\begin{equation}\label{box}
B_{l}({\eta})=
\left[\eta_1, \eta_1 + \frac{2d\zeta_l}{\Delta_l}\right)\times \cdots\times   \left[\eta_m, \eta_m + \frac{2d\zeta_l}{\Delta_l}\right) 
\subset \mathbb{R}^m
.
\end{equation}
We will need the following:
\begin{lem}\label{lem: Kolya}  
Let $R_l = \frac{dM_{l+1}}{\Delta_l} $. Then  for any $\eta$, the box $B_{l}({\eta})$ contains a point of the form $ Aq-p$ with 
\begin{equation}\label{qp}
 ||q|| \le  R_l 
,\,\,\,\,\,\,\
q\in \mathbb{Z}^n
,\,\,\,\,\, p \in \mathbb{Z}^m.
\end{equation}
\end{lem}
\begin{proof} 
 Fix $l$ 
 and consider parallelepiped 
$$
 \Pi_l =
 \{ z = (q,p) \in \mathbb{R}^d: ||q|| < M_{l+1}, ||Aq-p||< \zeta_l\}
 $$
Let  $\lambda_1\le \lambda_2\le ...\le \lambda_d$ be the successive minima  for the parallelepiped $\Pi_l$ with respect to the lattice $
 \mathbb{Z}^{d}$. As the volume of $\Pi_l $ is equal to $2^{d}\Delta_l$, by Minkowski theorem we know that 
 $$
 \lambda_1 \cdots \lambda_{d} \le \Delta_l^{-1}.
 $$
 But $\Pi_l$ does not contain non-trivial integer points and so $ \lambda_1 \ge 1$. We see that 
 $\lambda_{d}\le  \Delta_l^{-1}$ and
 $\lambda_{d} \cdot \Pi_l$
 contains $d$ independent integer points. So, the closed parallelepiped
 $$
 \wh{\Pi}_l = \text{closure of } \,d\Delta_l^{-1} {\Pi_l} \supset  d \lambda_{d} \Pi_l
 $$
 contains a fundamental domain of the lattice $\mathbb{Z}^{d}$. This means that any shift 
 \begin{equation}\label{formQ}
  \wh{\Pi}_l+w, \,\,\, w\in \mathbb{R}^{d}
  \end{equation} contains an integer point. But any parallelepiped
 $$
 \left\{ z = (q,p)\in \mathbb{R}^{d}:\,\,\,
 ||q|| \le R_l,\,\,\, Aq-p \in B_{l}({\eta})
\right\}
 $$
 is of the form (\ref{formQ}) which proves the Lemma.
 \end{proof}
 
 We continue with the proof of Theorem~\ref{thm: Kolya} by following the argument from \cite{Mosh-welldistributed} and choosing the parameters there more carefully.

 Assume now that the matrix $A$ belongs to $\cC$ and let $H_k$ be as in
 Definition~\ref{defn: Kolya}.
 Fix a positive $\varepsilon <\frac{1}{2}$ and define 
  a sequence  of  positive reals $
 \phi_k< 1
 $
 such that 
 \begin{equation}\label{ff1}
 \lim_{k \to \infty} \phi_k  =0,
  \end{equation}
 
  \begin{equation}\label{ff2}
  \phi_k \ge  
  d^d \left(\frac{3}{\varepsilon}\right)^{m} H_k
  \end{equation}
and
 \begin{equation}\label{ff3}
 \sum_{k=1}^\infty \phi_k 
\Delta_{l_k}^{d-1}=\infty.
 \end{equation}
 This is possible because by item \eqref{item: Kolya2} of Definition~\ref{defn: Kolya}  the right hand side of \eqref{ff2} tends to zero, and as we have the divergence in item~\eqref{item: Kolya1}, $\phi_k$ can tend to zero slow enough to satisfy divergence condition \eqref{ff3}.
 Define
 $$
 W_k = \left\lceil \frac{\Delta_{l_k}}{2d\zeta_{l_k}}\right\rceil^m
.
$$
Note that 
as $\frac{M_{l_{k}+1}}{\Delta_{l_{k}}} \to \infty$ as  $k\to \infty$, it follows from 
item \eqref{item: Kolya2} of Definition~\ref{defn: Kolya} that
\begin{equation}
\label{iiq}
\frac{\zeta_{l_{k}}}{\Delta_{l_{k}}}   \to 0, \,\,\ k \to \infty,
\end{equation}
and so using (\ref{iiq}) we see that  $W_k \to +\infty,$ as 
$k \to \infty$.
We cover the cube $[0,1)^m$ by  $ W_k$    boxes 
$B_{l_k}({\eta}_{i_1,...,i_n})$
of the form (\ref{box})
with 
$${\eta}_{i_1,...,i_n} = \left(\frac{i_1d\zeta_{l_k}}{\Delta_{l_k}},..., \frac{i_nd\zeta_{l_k}}{\Delta_{l_k}}\right), \,\,\,\, 0\le i_1,...,i_n <  \left\lceil \frac{\Delta_{l_k}}{2d\zeta_{l_k}}\right\rceil,\,\,\, i_1,...,i_n \in \mathbb{Z}
$$
which have disjoint interiors. 
 
We note that 
at least $W_k' = \left(  \left\lceil \frac{\Delta_{l_k}}{2d\zeta_{l_k}}\right\rceil -1\right)^m\sim W_k$ of these boxes are 
contained in the cube $[0,1)^m$.
Moreover, Lemma~\ref{lem: Kolya} shows that 
  in each of these boxes there is a point of the form $Aq-p$  satisfying (\ref{qp}) with $l = l_k$.

Put
$$
\delta_k = \frac{\phi_k^{\frac{1}{m}} }{R_{l_k}^{\frac{n}{m}}}  =  \phi_k^{\frac{1}{m}}\cdot\left( \frac{ \Delta_{l_k}}{dM_{l_k+1}}\right)^{\frac{n}{m}}.
$$
It is clear that 
$
\delta_k \le \phi_k^{\frac{1}{m}}\cdot \frac{\zeta_{l_k}}{\Delta_{l_k}}
$
because from (\ref{mi1}) we deduce
$
\Delta_{l_k}^d \le \Delta_{l_k} = M_{l_k+1}^n \zeta_{l_k}^m,
$
and this gives
$
\left(\frac{\Delta_{l_k}}{dM_{l_k+1}}\right)^{\frac{n}{m}} < \frac{\zeta_{l_k}}{\Delta_{l_k}}.
$
 
As in the proof from \cite{Mosh-welldistributed}
 we take 
 $$
 W_k'' = \left[ \frac{\left\lceil \frac{\Delta_{l_k}}{2d\zeta_{l_k}}\right\rceil-1}{3}\right]^m \sim \frac{W_{k}'}{3^m} \sim\frac{W_k}{3^m}
 $$
  boxes of the form
 \begin{equation}\label{l}
I_i^{[k]}=  \left[\xi_{1,i}^{[k]}- \delta_k ,\xi_{1,i}^{[k]}+\delta_k \right]\times\cdots\times
 \left[\xi_{m,i}^{[k]}-\delta_k ,\xi_{m,i}^{[k]}+\delta_k \right],\,\,\,\,1\le i \le W_k''
 \end{equation}
 with the  centres at certain points
 $$
 \xi_i^{[k]} = (\xi_{1,i}^{[k]}, ...,\xi_{m,i}^{[k]}) = Aq_i^{[k]}-p_i^{[k]},\,\,\,
 q_i^{[k]}\in \mathbb{Z}^n,
\, p_i^{[k]}\in \mathbb{Z}^n 
  ,\,\,\,\,\,|q_i^{[k]}|\le R_{l_k},\,\,\, 1\le i \le W_k''
  $$
 which belong to boxes 
 \begin{equation}\label{Cq}
 B_{l_k}({\eta}_{i_1,...,i_n})\,\,\,\,\,\text{with} \,\,\,\,\,i_1\equiv ...\equiv i_n \equiv 1 \pmod{3}. 
 \end{equation}
 In each of such boxes $B_{l_k}({\eta}_{i_1,...,i_n})$ we take just one point  ${\xi}_{i} $.
Then
\begin{equation}\label{aint}
I_i^{[k]}\cap I_{i'}^{[k]}=\varnothing,\,\,\,i\neq i'.    
 \end{equation}
 Consider two values  $k_1>k$.
 Recall that $I_i^{[k]}$ is a $m$-dimensional box with edge $ 2\delta_k$.  
 The centres of  the boxes $I_i^{[k_1]}$ 
 by the construction are
 distributed in $I_i^{[k]}$ uniformly and so
 \begin{equation}\label{sharp0}
 \sharp \{ i_1:\, 1\le i_1\le W_{k_1}'',\,\,\,  I_{i_1}^{[k_1]} \cap I^{[k]}_i \neq \varnothing \}  \le
 \left(2\delta_k/
 \left(6\cdot \frac{\zeta_{l_{k_1}}}{\Delta_{l_{k_1}}}\right)+1\right)^m .
 \end{equation}
We note that by condition (\ref{ff2}) we have
\begin{equation}\label{dop}
1\le  \varepsilon \cdot
2\delta_k/
 \left(6\cdot \frac{d\zeta_{l_{k_1}}}{\Delta_{l_{k_1}}}\right)
 \end{equation}
 because as $k_1 \ge k+1$ we have
  $$
  \delta_k  = \phi_k^{\frac{1}{m}} \cdot \left(\frac{\Delta_{l_k}}{dM_{l_k+1}}\right)^{\frac{n}{m}} \ge d^{\frac{d}{m}} \, \frac{3}{\varepsilon}\, 
  \frac{\zeta_{l_{k_1}}}{\Delta_{l_{k_1}}}.
  $$

Recall that  $\lambda (I_i^{[k]}) = (2\delta_k)^m  $ and
$ W_{k_1}'' \sim \frac{1}{3^m}\cdot \left( \frac{\Delta_{l_{k_1}}}{2d\zeta_{l_{k_1}}}\right)^m$.
So for $k$ large enough
 we deduce from (\ref{sharp0})  and (\ref{dop}) an upper bound
 \begin{equation}\label{sharp}
  \sharp \{ i_1:\, 1\le i_1\le W_{k_1}'',\,\,\,  I_{i_1}^{[k_1]} \cap I^{[k]}_i \neq \varnothing \}   \le
  \lambda (I_i^{[k]}) W_{k_1}'' (1+2^m\varepsilon).
 \end{equation}
 Now we consider the union
 $$
 E_k = \bigcup_{i=1}^{W_k''} I_i^{[k]}.
 $$
 By (\ref{aint})  we get
 $$
 \lambda (E_k)=  \sum_{i} \lambda (   I_i^{[k]})\asymp
 \delta_k^m W_k\asymp
 \phi_k 
 \Delta_{l_k}^{d-1}
 $$
 and so 
by (\ref{ff3}) we see that 
  \begin{equation}\label{ccc1}
  \sum_{k=1}^\infty
 \lambda (E_k) = \infty.
 \end{equation}
 Then from (\ref{sharp}) for $ k_1> k$ we have
   \begin{equation}\label{ccc2}
 \lambda (E_k\cap E_{k_1} )\le (1+ 2^m\varepsilon)  \lambda (E_k)  \lambda (E_{k_1})
.
 \end{equation} 
 It follows from \cite[Theorem 18.10]{Gut}
 that for the set
 $$
  E = \{ {\eta}:\,\, \exists \,\,\text{infinitely many}\,\, k\,\,
  \text{such that }\,\, {\eta} \in E_k\}
  $$
  we have
  $$
  \lambda (E)
  \ge
  \limsup_{t\to \infty}
  \frac{\left(\sum_{k=1}^t   \lambda (E_k) \right)^2}{  \sum_{k,k_1 =1}^t \lambda (E_k \cap E_{k_1})} \ge \frac{1}{1+ 2^m\varepsilon},
  $$
  by (\ref{ccc1}) and (\ref{ccc2}).   As $\varepsilon $ is arbitrary, we see that  $
  \lambda (E) = 1$.
  
  Finally, if $\eta \in E$ then there exists infinitely many $k$ such that for each $k$  there exists
  $ q\in \mathbb{Z}^n$ with $ ||q||\le R_{l_k}$ and $p\in \mathbb{Z}^n$ such that 
  $ ||Aq-p-\eta|| \le \delta_k$. In particular, from the definition of $R_{l_k}$    for these $(p,q)$ we get
  $$
   ||Aq-p-\eta|| \cdot ||q||^{\frac{n}{m}}\le
   \delta_k R_{l_k}^{\frac{n}{m}}=
    \phi_k^{\frac{1}{m}}.
   $$
We see that  condition (\ref{ff1}) leads to the inclusion 
  ${\rm Bad}_A\cap [0,1)^m \subset [0,1)^m \sm E$ which finishes the proof.
   \end{proof}

\section{Appendix: Roy's parametric geometry of numbers}\label{sec: Roy}

As it was mentioned before, the existence results from Propositions \ref{prop: special singular form} and \ref{abcde} as well as the main result from the paper 
\cite{LRSY}  can be deduced by means of existence results from parametric geometry of numbers  obtained in \cite{Roy-MathZ} and \cite{DFSU}.
In particular, here in Appendix we  show how one can get Proposition \ref{prop: special singular form}.

 \begin{defn}[Log-minima functions]\label{defn: log minima}
     Let $\theta = (\theta_1, \theta_2) \in \R^2$. For $Q \in \R$ with $Q \geq 1$, we define the convex body
     \begin{equation*}
         \mathcal{C}(Q) := \left\lbrace v=(v_1, v_2, v_3) \in \R^3: \|v\| \leq 1,\ |v_1\theta_1 + v_2 \theta_2 + v_3| \leq Q^{-1}\right\rbrace
     \end{equation*} 
     For $j=1,2,3$, we define the successive minima functions $\lambda_j:[1,\infty) \to \R$:
     \begin{equation*}
         \lambda_j(Q) := \min\left\lbrace r\in \R: r\mathcal{C}(Q) \text{ contains $j$ independent vectors of $\Z^3$}\right\rbrace
     \end{equation*}
     and the log-minima functions $L_j: [0,\infty) \to \R$:
     \begin{equation*}
         L_j(q) = \log (\lambda_j(e^q)).
     \end{equation*}
 \end{defn}
The relation between the first log-minima function and the irrationality measure function for a vector $\theta \in \R^2$ is given by the following Proposition.
\begin{prop}\label{prop: L1 > unbounded}
    Let $\theta \in \R^2$. Let $L_1$ be the first associated log-minima function and let $\Psi$ be the associated irrationality measure function.
    If, for some constants $0 < A < 1$ and $0 < B$, the equation
    \begin{equation}\label{eq: Aq - B}
        L_1(q) > Aq - B
    \end{equation}
    is satisfied for an unbounded set of  $q \in \R_{\geq 0}$,
    then the equation
    \begin{equation*}
        \Psi(t) \geq e^{\frac{-B - \ln(1 + \|\theta\|)}{A}}t^{1 - \frac{1}{A}}
    \end{equation*}
    is satisfied for an unbounded set of  $t \in \R_{\geq 0}$.
\end{prop}
\begin{proof}
    If, for some $q >0$, \eqref{eq: Aq - B} holds, then from the definitions we have
    \begin{equation}\label{eq: Aq-B set}
        \left\lbrace v \in \R^3: \|v\| \leq e^{Aq-B},\ |v_1\theta_1 + v_2 \theta_2 + v_3| \leq e^{Aq - B }e^{-q}\right\rbrace \cap \Z^3 = \{0\}.
    \end{equation}
    Now let $C = \ln(1 + \|\theta\|)$. Suppose, by way of contradiction, that 
\begin{equation*}
    \Psi(e^{Aq- B - C})  < e^{Aq-B - C}e^{-q}.
\end{equation*}
This gives the existence of $a = (a_1, a_2) \in \Z^2$ with $0 < \|a\| \leq e^{Aq-B - C}$ and some $b \in \Z$ for which
\begin{equation}\label{eq: atheta + b}
    |a_1\theta_1 + a_2 \theta_2 + b| < e^{Aq- B - C} e^{-q}.
\end{equation}
Thus,
\begin{equation}\label{eq: |b|}
\begin{split}
    |b|  &= |b + (a_1\theta_1 + a_2 \theta_2) - (a_1 \theta_1 + a_2 \theta_2)| \\
    &< e^{Aq - B - C} e^{-q} + \|\theta\| e^{Aq - B - C} \\
    & \leq e^{Aq - B} e^{-C} (1 + \|\theta\|) = e^{Aq - B}.
\end{split}
\end{equation}
Since $C \geq 0$, we have
\begin{equation*}
    \max\{|a_1|, |a_1|, |b|\} \leq e^{Aq - B}\ \text{ and }\ |a_1 \theta_1 + a_2 \theta_2 + b| \leq e^{Aq - B} e^{-q}
\end{equation*}
which contradicts \eqref{eq: Aq-B set}.
Thus, we must have
\begin{equation*}
    \Psi(e^{Aq-B -C}) \geq e^{Aq - B - C} e^{-q}.
\end{equation*}
If we substitute $t = e^{Aq -B -C}$, we get
\begin{equation*}
    \Psi(t) \geq e^{\frac{-B -C}{A}}t^{1-\frac{1}{A}}
\end{equation*}
which is the required result.
\end{proof}
\begin{prop}\label{prop: L1 leq all}
    Let $\theta \in \R^2$, let $L_1$ be the associated first log-minima function and let $\Psi$ be the associated irrationality measure function.
    Suppose, for some constants $0 < A < 1$ and $0 < B$, we have
    \begin{equation}\label{eq: L1 linear}
        L_1(q) \leq Aq + B \text{ for all sufficiently large } q \in \R_{\geq 0}.
    \end{equation}
    Then the irrationality measure function satisfies
    \begin{equation*}
        \Psi(t) \leq e^{B/A} t^{1- \frac{1}{A}} \text{ for all  $t \in \R_{\geq 0}$ sufficiently large.}
    \end{equation*}
\end{prop}
\begin{proof}
    This is again a change of variables. The equation \eqref{eq: L1 linear} implies that, for all large $q$,
    \begin{equation*}
        e^{Aq + B} \mathcal{C}(e^q) \cap \Z^3 \neq \{0\}.
    \end{equation*}
    Unwinding the definitions, we obtain
    \begin{equation*}
        \Psi(e^{Aq+B}) \leq e^{Aq +B} e^{-q} \text{ for all sufficiently large } q.
    \end{equation*}
    Putting $t = e^{Aq+B}$, we get the required conclusion.
\end{proof}
The log-minima functions $(L_1,L_2,L_3)$ are modeled by the following functions.
\begin{defn}[\cite{Roy-MathZ}, Definition 4.1 of $n$-systems]\label{defn: 3-system}
    Let $I \subset [0,\infty)$ be a subinterval. A $3$-system on $I$ is a continuous piecewise linear map $P=(P_1,P_2,P_3):I \to \R^3$ such that 
    \begin{enumerate}
        \item[(S1)] For each $q \in I$, we have
        \begin{equation*}
            0 \leq P_1(q) \leq P_2(q) \leq P_3(q) \text{ and } P_1(q) + P_2(q) + P_3(q) = q.
        \end{equation*}
        \item[(S2)] If $H$ is a nonempty open subset on which $P$ is differentiable, then there is an integer $r$ with $1\leq r \leq 3$ such that $P_r$ has slope $1$ on $H$ while the other components of $P$ are constant.
        \item[(S3)] If $q$ is an interior point of $I$ where $P$ is not differentiable and if the integers $r$, $s$ for which $P'_r(q^-) = P'_r(q^+)=1$ also satisfy $r < s$, then we have
        \begin{equation*}
            P_r(q) = P_{r+1}(q) = \dots = P_s(q).
        \end{equation*}
    \end{enumerate}
\end{defn}
\begin{rem}
    Piecewise linear means the points of $I$ where $P$ is not differentiable has discrete closure in $\R$ and $P$ is linear on each connected component of the complement (in $I$) of the non-differentiable points. The notation $P_r'(q^-)$ and $P_s'(q^+)$ denote the left and right derivatives respectively (assuming $q$ is not in the boundary of $I$).
\end{rem}
We have the following fundamental theorem.
\begin{thm}[\cite{Roy-MathZ} Theorem 4.2]\label{thm: Roy}
    For each $3$-system $P$ on an interval $[q_0, \infty)$, there is a $\theta \in \R^2$ such that, considering the log-minima function $L=(L_1,L_2,L_3): [q_0, \infty) \to \R^3$ associated to $\theta$, we have
    \begin{equation*}
        L - P \text{ is bounded on } [q_0, \infty).
    \end{equation*}
\end{thm}
We are finally ready for the
\begin{proof}[Proof of Proposition \ref{prop: special singular form}]
    Let $Q \in \R$ satisfy $Q >2$. Consider the following graphs of three piecewise linear functions on the interval $[1,Q-1]$.
    
    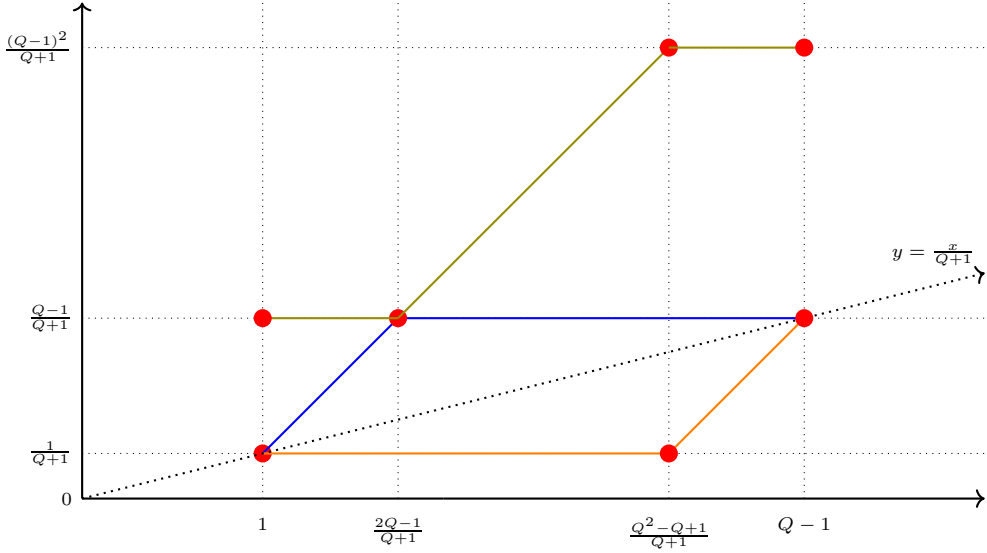
\begin{figure}[H]
    \begin{center}
    \begin{tikzpicture}[scale=1.2, baseline=(current  bounding  box.center)]
    
    \draw [thick, ->] (-5,-2) -- (-5,3.5);
    \node [left] at (5,0.7) {\tiny $y = \frac{x}{Q+1}$};
    \draw [thick] (-5,-2) -- (-3,-2);
    \draw [thick, ->] (-1,-2) -- (5,-2);

    \draw[thick] (-3,-2)--(-1,-2);

    \node [left] at (-5,-2) {\tiny $0$};
    
    \node [below] at (-3,-2.1) {\tiny $1$};
    \draw [dotted] (-3,-2) -- (-3,3.5);

    \node [below] at (-1.5,-2.1) {\tiny $\frac{2Q -1}{Q+1}$};
    \draw [dotted] (-1.5,-2) -- (-1.5,3.5);

    \node [below] at (1.5,-2.1) {\tiny $\frac{Q^2 - Q + 1}{Q+1}$};
    \draw [dotted] (1.5,-2) -- (1.5,3.5);

    \node [below] at (3,-2.1) {\tiny $Q-1$};
    \draw [dotted] (3,-2) -- (3,3.5);

    \node [left] at (-5,-3/2) {\tiny $\frac{1}{Q+1}$};
    \draw [dotted] (-5,-3/2) -- (5,-3/2);

    \node [left] at (-5,0) {\tiny $\frac{Q-1}{Q+1}$};
    \draw [dotted] (-5,0) -- (5,0);

    \node [left] at (-5,3) {\tiny $\frac{(Q-1)^2}{Q+1}$};
    \draw [dotted] (-5,3) -- (5,3);

    \fill [red] (-3,-3/2) circle[radius=0.1];

    \draw [thick, blue] (-3,-1.5) -- (-1.5,0);
    \draw [thick, blue] (-1.5,0) -- (3,0);
    \fill[red] (-1.5,0) circle[radius=0.1];

    \draw[thick, olive] (-3,0)--(-1.5,0);
    \fill[red] (-3,0) circle[radius=0.1];
    \draw[thick, olive] (-1.5,0) -- (1.5,3);
    \fill[red] (1.5,3) circle[radius=0.1];
    \draw[thick, olive] (1.5,3) -- (3,3);
    \fill[red] (3,3) circle[radius=0.1];

    \draw [thick, orange] (-3,-1.5)--(1.5,-1.5);
    \draw [thick, orange] (1.5,-1.5)--(3,0);

    \fill [red] (1.5,-1.5) circle[radius=0.1];

    \draw[thick, dotted, ->] (-5,-2) -- (5,0.5);

    \fill [red] (3,0) circle[radius = 0.1];
    \end{tikzpicture}
    \end{center}
    \caption{The orange, blue are olive graphs are called $P_1$, $P_2$, $P_3$ respectively.}
\end{figure}
Here is the precise definition of the function $P:[1,Q-1] \to \R^3$:
\begin{equation}\label{eq: I_0 defn}
    (P_1(q), P_2(q), P_3(q)) := \begin{cases}
			  \left(\frac{1}{Q+1},\ \frac{1}{Q+1} + q-1,\ \frac{Q-1}{Q+1} \right) & ; 1\leq q \leq \frac{2Q-1}{Q+1} \\ \medskip
     
           \left(\frac{1}{Q+1},\ \frac{Q-1}{Q+1},\ \frac{Q-1}{Q+1} + q - \frac{2Q-1}{Q+1} \right)   & ; \frac{2Q-1}{Q+1}  \leq q \leq \frac{Q^2 - Q + 1}{Q+1}  \\
              \smallskip 
              
            \left(\frac{1}{Q+1} + q - \frac{Q^2 - Q + 1}{Q+1},\ \frac{Q-1}{Q+1},\ \frac{(Q-1)^2}{Q+1}\right)     & ; \frac{Q^2 - Q + 1}{Q+1}  \leq q \leq Q-1
		  \end{cases}.
\end{equation}
We leave it to the reader to check that $P$ satisfies Definition \ref{defn: 3-system} on $[1,Q-1]$.
We now extend this function to a $3$-system on all of $[1,\infty)$;
We let $I_0 := [1,Q-1]$. For each $l \in \N$, we inductively define the interval $I_l$ as the closed interval of length $(Q-1)^l \times  \on{length}(I_{0})$ that is contiguous and to the right of $I_{l-1}$.
We let $f_l$ be the linear, orientation-preserving bijection from $I_l$ to $I_0$ and define
\begin{equation}\label{eq: P tilde}
    \widetilde{P}(q) := (Q-1)^l P(f_l(q)) \text{ for } q \in I_l.
\end{equation}
Note that $\widetilde{P}$ is well defined: For if we have $q \in I_{l-1}\cap I_l$,
\begin{equation*}
\begin{split}
    (Q-1)^lP(f_{l}(q)) &= (Q-1)^l P(1) \\
    &= (Q-1)^l\left( \frac{1}{Q+1}, \frac{1}{Q+1}, \frac{Q-1}{Q+1}\right) \\
    &= (Q-1)^{l-1}\left( \frac{Q-1}{Q+1}, \frac{Q-1}{Q+1}, \frac{(Q-1)^2}{Q+1}\right)\\
    &= (Q-1)^{l-1} P(Q-1) \\
    &= (Q-1)^{l-1} P(f_{l-1}(q)).
\end{split}
\end{equation*}
We have the following straightforward Claim which we write out completely. The reader might prefer to do it themselves.
\begin{claim}\label{claim: P inequalities}
    The function $\widetilde{P} = (\widetilde{P}_i)_{i=1,2,3}: [1,\infty) \to \R$ is a $3$-system with 
    \begin{enumerate}
        \item[(a)] $\widetilde{P}_1(q) \leq \frac{q}{Q+1}$ for all $q \geq 1$.
        \item[(b)] If $q$ is the right endpoint of $I_n$, then $\widetilde{P}_1(q) = \frac{q}{Q+1}$.
    \end{enumerate}
    \end{claim}
    \begin{proof}
        $\widetilde{P}$ is continuous and piecewise linear since it is so on each interval $I_l$. Moreover, the maps $f_l: I_l \to I_0$ have slopes $(Q-1)^{-l}$ which cancels with the scaling factor $(Q-1)^l$ of formula  \eqref{eq: P tilde}. Thus, Definition \ref{defn: 3-system}(S1, S2) become clear.
        We leave it to the reader to check Definition \ref{defn: 3-system}(S3) at the end points of $I_l$ for each $l\in \N$. Thus $\widetilde{P}$ is a $3$-system.

        For parts (a) and (b) we first prove that for each $l \in \N \cup \{0\}$ we have
        \begin{equation*}
            I_l = \left[(Q-1)^l, (Q-1)^{l+1} \right].
        \end{equation*}
        This is true for $I_0$ by definition. We let $l \geq 1$ and write $I_l = [a_n,b_n]$. We compute, using induction, that
        \begin{equation*}
            a_l = a_{l-1} + (Q-1)^{l-1} \on{length}(I_0) = (Q-1)^{l-1} + (Q-1)^{l-1}(Q-2) =  (Q-1)^l.
        \end{equation*}
        Similarly, for $b_n$ we have
        \begin{equation*}
            b_l = b_{l-1} + (Q-1)^l (Q-2) = (Q-1)^l + (Q-1)^l(Q-2) = (Q-1)^{l+1}.
        \end{equation*}

        Part (a) for $q \in I_0$ is clear from formula \eqref{eq: I_0 defn}. For  $l \in \N$ and $q \in I_l = [a_l,b_l]$, we can then compute
        \begin{equation*}
        \begin{split}
            \widetilde{P}_1(q) &= (Q-1)^lP_1(f_l(q)) \\
            &\leq (Q-1)^l \frac{f_l(q)}{Q+1}.
        \end{split}
        \end{equation*}
        Now $q \mapsto (Q-1)^l f_l(q)$ is a linear function of slope $1$ passing through the point
        \begin{equation*}
            (a_l, (Q-1)^l) = ((Q-1)^l, (Q-1)^l).
        \end{equation*}
        Thus, $(Q-1)^l f_l(q) = q$ on the domain of $f_l$ and so we have $\widetilde{P}_1(q) \leq q (Q+1)^{-1}$.

        For part (b), we have that 
        \begin{equation*}
            \widetilde{P}_{1}(b_l) = (Q-1)^lP_1(f_l(b_l)) = (Q-1)^l \frac{Q-1}{Q+1} = \frac{b_l}{Q+1}.
        \end{equation*}
    \end{proof}
We now apply Theorem \ref{thm: Roy} to the $3$-system $\widetilde{P}$ and obtain a $\theta \in \R^2$ with a constant $B >0$ such that
\begin{equation*}
    |L_1(q) - \widetilde{P}_1(q)| \leq B \text{ for all } q \in [1,\infty).
\end{equation*}
Here $L_1$ is the first log-minima function associated to $\theta$ as in Definition \ref{defn: log minima}. 
Claim \ref{claim: P inequalities} then shows
\begin{equation*}
    L_1(q) \leq \frac{q}{Q+1} + B \text{ for all } q \in [1,\infty)
\end{equation*}
and that 
\begin{equation*}
    -B -1 + \frac{q}{Q+1} < L_1(q) \text{ for an unbounded set of } q \in [1,\infty).
\end{equation*}
Applying Propositions \ref{prop: L1 > unbounded} and \ref{prop: L1 leq all}, we see the existence of constants $C_1, C_2$ for which
\begin{equation*}
    \Psi(t) \leq C_1 t^{-Q} \text{ for all sufficiently large } t\geq 1 
\end{equation*}
and \begin{equation*}
    C_2 t^{-Q} \leq \Psi(t) \text{ for an unbounded set of } t \geq 1.
\end{equation*}
Here, $\Psi$ is the irrationality measure function associated to $\theta$.
\end{proof}

\subsection*{Acknowledgements}
This work has received funding from the European Research Council (ERC) under the European Union’s Horizon 2020 Research and Innovation Program, Grant agreement no. 754475.
An essential part of this work is done during the first and second author's stay in Israel Institute of Technology (Technion). They thank the people from Technion for the extremely friendly atmosphere and wonderful opportunities for work.

\def\cprime{$'$}

\end{document}